\documentclass[11pt,a4paper, twoside, leqno]{arxiv}

\usepackage[english]{babel}
\usepackage[utf8]{inputenc}
\usepackage{fancyhdr}
\usepackage{indentfirst}
\usepackage{graphicx}
\usepackage{newlfont}
\usepackage[active]{srcltx}
\usepackage{mathrsfs}
\usepackage{amssymb}
\usepackage{amsmath}
\usepackage{latexsym}
\usepackage{amsthm}
\usepackage[all]{xy}
\usepackage{booktabs}
\usepackage{multicol}
\numberwithin{equation}{section}

\title{Derived Categories of Toric Fano 3-Folds via the Frobenius Morphism}
\titlemark{Derived Categories of Toric Fano 3-Folds}

\titlenote{This paper is the solution of one of the problems proposed during the 2009 edition of
P.R.A.G.MAT.I.C. summer school}


\author{Alessandro Bernardi}
\address{
Dipartimento di Matematica\\
Universit\`a  degli Studi di Firenze\\
\email{bernardi@math.unifi.it}
}
\author{Sofia Tirabassi}
\address{
Dipartimento di Matematica\\
Universit\`a  degli Studi  Roma Tre\\
\email{tirabass@mat.uniroma3.it}
}

\date{2009/12/20}

\newcommand{\occ}{\'{o}\;}

\newcommand{\bbP}{\mathbb{P}}

\newcommand{\R}{\mathbb{R}}
\newcommand{\C}{\mathbb{C}}

\newcommand{\Z}{\mathbb{Z}}
\newcommand{\N}{\mathbb{N}}

\newcommand{\inv}[1]{{#1}^{-1}}

\newcommand{\ra}{\rightarrow}

\newcommand{\dual}{\vee}

\newcommand{\spec}{\mathrm{Spec\:}}

\newcommand{\fas}[1]{\mathcal{#1}}

\newcommand{\F}{\mathscr{F}}
\newcommand{\strutt}[1]{\mathcal{O}_{#1}}
\newcommand{\ox}{\strutt{X}}

\newcommand{\sO}{\mathscr{O}}

\newcommand{\Hom}{\mathrm{Hom}}
\newcommand{\Ext}{\mathrm{Ext}}
\newcommand{\Pic}{\mathrm{Pic}}
\newcommand{\ellesigma}[2]{l_{\sigma_{#1}}(v_{#2})}
\newcommand{\elsig}[1]{l_{\sigma_{#1}}}
\newcommand{\ei}[1]{\widehat{e}_{#1}}

\newcommand{\duno}{\mathcal{D}_1}
\newcommand{\ddue}{\mathcal{D}_2}

\newcommand{\euno}{\mathcal{E}_1}
\newcommand{\edue}{\mathcal{E}_2}

\newcommand{\equat}{\mathcal{E}_4}

\newcommand{\odiv}[1]{\fas{O}(\pm(#1))}
\DeclareMathOperator{\gl}{GL}

\theoremstyle{plain}                    
\newtheorem{conj}[thm]{Conjecture}
\newtheorem*{teo*}{Theorem}

\theoremstyle{definition}               
\newtheorem{defin}{Definition}[section]
\newtheorem{nota}{Notation}[section]

\theoremstyle{remark}                   
\newtheorem{prob}{Problem}

\begin{document}

\maketitle

\begin{abstract}
In \cite[Conjecture 3.6]{frobenius}, Costa and Mir\'{o}-Roig state the following conjecture:\\
Every smooth complete toric Fano variety has a full strongly exceptional collection of line bundles. The goal
of this article is to prove it for toric
Fano 3-folds.

\end{abstract}

\section{Introduction}
Let $X$ be a smooth projective variety defined over an algebraically closed field $K$ of characteristic $0$
and we denote by
$D^b(X)=D^b(\fas{O}_X-mod)$ the derived category of bounded complexes of coherent sheaves of
$\fas{O}_X-$modules. Recall that a coherent sheaf $\F$ on
a smooth projective variety $X$ is \emph{exceptional} if $\mathrm{Hom}_{\fas{O}_X}(\F,\F)=K$ and
$\mathrm{Ext}^i_{\fas{O}_X}(\F,\F)=0$ for $i>0$. An
ordered collection $(\F_0,\F_1,...,\F_n)$ of coherent sheaves on $X$ is called an \emph{exceptional}
collection if each sheaf $\F_i$ is
exceptional and $\Ext^i_{\ox}(\F_k,\F_j)=0$ for $j<k$ and $i \geq0$; moreover it is called a \emph{strongly
exceptional} collection if each $\F_i$ is
exceptional, $\mathrm{Hom}_{\fas{O}_X}(\F_k,\F_j)=0$ for $j<k$ and $\Ext^i_{\fas{O}_X}(\F_j,\F_k)=0$ for
$i\geq 1$ and for all $j,k$. A strongly
exceptional collection $(\F_0,\F_1,...,\F_n)$ of coherent sheaves on $X$ is called \emph{full} if
$\F_0,\F_1,...,\F_n$ generate the bounded derived
category $D^b(X)$. 

We are interested to study the following general problem:  

\begin{prob}
   To characterize smooth projective varieties $X$ which have a full strongly exceptional collection of
coherent sheaves and, even more, if there is
one made up of line bundles. 
  \end{prob}

In this context there is an important conjecture \cite{king} due to King :
\begin{conj}
Every smooth complete toric variety has a full strongly exceptional collection of lines bundles. 
\end{conj}
Kawamata \cite{kawamata} proved that the derived category of a smooth complete toric variety has a full
exceptional collection of objects, but the
objects in these collections are sheaves rather than line bundles and the collection is only exceptional and
not strongly exceptional.
In the toric context, there are many contributions to the above conjecture. It turns out to be true for
projective spaces \cite{beilinson},
multiprojective spaces \cite[Proposition 4.16]{tilting}, smooth complete toric varieties with Picard number
$\leq2$ \cite[Corollary 4.13]{tilting} and
smooth complete toric varieties with a splitting fan \cite[Theorem 4.12]{tilting}.

Recently, in \cite{counterexample}, Hille and Perling had constructed a counterexample at the King's
Conjecture, precisely an example of smooth non Fano
toric surface which does not have a full strongly exectional collection made up of line bundles. 
In the Fano context, there are some numerical evidences (see \cite{smallpicard}) which give support to the
following conjecture due to Costa and
Mir\'{o}-Roig \cite[Conjecture 3.6]{frobenius}:
\begin{conj}
Every smooth complete Fano toric variety has a full strongly exceptional collection of line bundles.
\end{conj}
But the Fano hypotesis is not necessary, in fact, in Theorem 4.12 \cite{tilting}, Costa and Mir\'{o}-Roig
constructed full strongly exceptional
collections of line bundles on families of smooth complete toric varieties none of which is entirely of Fano
varieties.

In this article we will prove the above conjecture for Fano toric 3-folds, more specifically:
\begin{teo*}[Main Theorem \ref{mainthm}]
All smooth toric Fano 3-folds have a full strongly exceptional collection made up of line bundles.
\end{teo*}

In order to prove the above theorem, we will refer to the classification of Fano toric 3-folds due to Batyrev
(\cite{batyrev}). In this classification
there are some known case, i.e. some class of Fano toric 3-folds are just known to have a full strongly
exceptional collection made up of line
bundles.
For each unknown case we will use a method \cite{bondal} due to Bondal and an algorithm \cite{thomsen} by
Thomsen to produce a candidate to became a
full strongly exceptional collection. Then using vanishing theorems we will check that these candidates are
indeed a full strongly exceptional
collection.

\section{The Classification of Fano Toric 3-Folds}
In this section we briefly recall the usual notation and terminology for toric varieties and we present the
classification of  toric Fano 3-folds, due
to Batyrev (\cite{batyrev}).
\begin{defin}
A complete \emph{toric variety} of dimension $n$ over an algebraically closed field $K$ of characteristic $0$
is a smooth variety $X$ that contains a
torus $T=(\C^*)^n$ as a dense open subset, together with an action of $T$ on $X$, that extends the natural
action of $T$ on itself. Let $N$ be a
lattice in $\R^n$, by a \emph{fan} $\Sigma$ of strongly convex polyhedral cones in $N_{\R}:=N\otimes_{\Z} \R$
is meant a set of rational strongly
convex polyhedral cones $\sigma$ in $N_{\R}$ such that:
\begin{enumerate} 
\item each face of a cone in $\Sigma$ is also a cone in $\Sigma$;
\item the intersection of two cone in $\Sigma$ is a face of each.
\end{enumerate}
We will assume that a fan is finite.
\end{defin}

It is known that a complete toric variety $X$ is characterized by a fan\linebreak $\Sigma:=\Sigma(X)$. Let
$M:=\Hom_{\Z}(N,\Z)$ be the dual lattice of
$N$, if
$e_0,...,e_{n-1}$ is the basis of $N$, we indicate by $\hat{e}_0,...,\hat{e}_{n-1}$ the dual base. If $\sigma$
is a cone in $N$, the dual cone
$\sigma^{\dual}$ is the subset of $M$ defined as:
$$\sigma^\vee:=\{\eta\in M\:|\:\eta(v)\geq 0\;\text{for all $v\in N$}\}.$$
So we have a commutative semigroup
$S_{\sigma}:=\sigma^{\dual}\cap M$ and an open affine toric subvariety $U_{\sigma}:=\spec(K[S_{\sigma}])$. 

From a fan $\Sigma$, the toric variety $X(\Sigma)$ is constructed by taking the disjoint union of the affine
toric subvarieties $U_{\sigma}$, one for
each $\sigma\in \Sigma$ and gluing.
Conversely, any toric variety $X$ can be realized as $X(\Sigma)$ for a unique fan $\Sigma$ in $N$.
The toric variety $X(\Sigma)$ is smooth if and only if any cone $\sigma\in \Sigma$ is generated by a part of a
basis of $N$.
\begin{defin}
We put $\Sigma(i):=\{\sigma\in \Sigma\mid \dim \sigma=i \}$, for any $0\leq i\leq n$. There is a unique
generator $v\in N$ for any 1-dimensional cone
$\sigma\in \Sigma(1)$ such that\linebreak $\sigma\cap N=\Z_{\geq 0}\cdot v$ and it is called a \emph{ray
generator}. There is a one-to-one
correspondence between the ray generators $v_1,...,v_k$ and the toric divisors $D_j$. Moreover, $D_1\cap
D_2\cap ... \cap D_k=0$ if and only if the
corresponding vectors $v_1, v_2,...,v_k$ span a cone in $\Sigma$. 
\end{defin}
If $X$ is a smooth toric variety of dimension $n$ and $m$ is the number of toric divisor of $X$ (hence $m$ is
also the number of $1-$dimensional rays
generator in $\Sigma(X)$) then we have an exact sequence of $\Z-$modules:
\begin{displaymath}
\xymatrix{
0 \ar[r] & M \ar[r] & \Z^m \ar[r] & \Pic(X) \ar[r] & 0}.
\end{displaymath}
Therefore we have that the Picard number of $X$ is $\rho=m-n$ and the anticanonical divisor is given by
$-K_X=D_1+...+D_m$. The relations among the
toric divisors are given by $\sum^m_{i=1}<u,v_i>D_i=0$ for $u$ in a basis of $M=\Hom(N;\Z)$. A smooth toric
Fano variety $X$ is a smooth toric variety
with anticanonical divisor $-K_X$ ample.
\begin{defin}
A set of toric divisors $\{D_1,...,D_k\}$ on $X(\Sigma)$ is called a \emph{primitive
set} if $D_1\cap...\cap D_k=\emptyset$ but $D_1\cap...\cap\hat{D_j} \cap...\cap D_k\neq \emptyset$ for all
$j$, with\linebreak $1\leq j\leq k$.
Equivalently, this means $< v_1,..., v_k >\notin \Sigma$ but $< v_1,...,\hat{v_j},..., v_k >\in \Sigma$ for
all $j$, with $1\leq j\leq k$, and $P =
\{v_1,..., v_k\}$ is called a \emph{primitive collection}. If $S := \{D_1,...,D_k\}$ is a primitive set, the
element $v := v_1 + ... + v_k$ lies in the relative interior of a unique cone of $\Sigma$, if the cone is
generated by ${v'}_1,..., {v'}_s$, then $v_1
+ ... + v_k =a_1{v'}_1+...+a_s{v'}_s$ with $a_i > 0$ is the corresponding \emph{primitive relation}.
\end{defin}
In terms of primitive collections and relations we have a nice criterion for checking if a smooth toric
variety is Fano or not. A smooth toric variety
$X(\Sigma)$ is Fano if and only if for every primitive relation $$v_{i_1} +...+ v_{i_k}-c_1v_{j_1}-...- c_r
v_{j_r} = 0$$ one has $k -\sum^r_{i=1}c_i
> 0$.
We racall that if we indicate by $K_0(X)$ Grothendieck group, we known that its rank is equal to the
number of maximal cones of $\Sigma$, in
particular for a smooth toric Fano $3-$folds $X(\Sigma)$ we can calculate it only knowing the number
$\upsilon$ of vertices of $\Sigma$ i.e.
$rank(K_0(X(\Sigma)))=2\upsilon-4$ and $\rho=\upsilon-3$, where $\upsilon$ is the number of ray generators
of the toric variety $X$. 

\pagebreak
In the table below we give the list of $3$-dimensional toric Fano varieties $V=V(P)$. We denote by $S_i$ the
Del Pezzo surface obtained blowing up of
$i$ points on $\mathbb{P}^2$. The numbers $\upsilon, \rho$ and $k_0$ denote respectively the number of
vertices, the Picard number and the rank of the
Grothendieck group.

\begin{center}
\begin{tabular}{llccc}
\toprule
  &Class of Toric Fano 3-folds  & $\upsilon$ & $\rho$ & $k_0$ \\
\midrule
Type I & $\bbP^3$ & 4 &  1 & 4  \\
\midrule
&$\bbP_{\bbP^2}(\sO \oplus \sO(2))$ & 5 & 2 & 6 \\
 &$\bbP_{\bbP^2}(\sO \oplus \sO(1))$ & 5 & 2 & 6 \\
Type II &$\bbP_{\bbP^2}(\sO \oplus \sO \oplus \sO(1))$ & 5 & 2 & 6 \\
 &$\bbP^2\times \bbP^1$ & 5 & 2 & 6 \\
 &$\bbP_{\bbP^1\times \bbP^1}(\sO\oplus \sO(1,1))$ & 6 & 3 & 8 \\
 &$\bbP_{S_1}(\sO\oplus \sO(l))$, $l^2=1$ on $S_1$ & 6 & 3 & 8 \\
\midrule 
&$\bbP^1\times \bbP^1 \times \bbP^1$ & 6 & 3 & 8 \\
 &$S_1\times \bbP^1 $ & 6 & 3 & 8 \\
Type III &$\bbP_{\bbP^1\times\bbP^1}(\sO\otimes \sO(1,-1))$& 6 & 3 & 8 \\
 &$S_2\times \bbP^1$& 7 & 4 & 10 \\
 &$S_3\times \bbP^1$& 8 & 5 & 12 \\
\midrule
&$Bl_{\bbP^1}(\bbP_{\bbP^2}(\sO \oplus \sO(1)))$& 6 & 3 & 8 \\
 &$Bl_{\bbP^1}(\bbP^2\times \bbP^1)$& 6 & 3 & 8 \\
 Type IV 
&$S_2-$bundle over $\bbP^1$ & 7 & 4 & 10 \\
& $S_2-$bundle over $\bbP^1$ & 7 & 4 & 10 \\
 &$S_2-$bundle over $\bbP^1$& 7 & 4 & 10 \\
\midrule
 Type V
 &$S_3-$bundle over $\bbP^1$& 8 & 5 & 12 \\
\bottomrule
\end{tabular}
\end{center}

The following theorem is due to Costa and Mir\occ-Roig (see Theorem 4.21 \cite{tilting}):
\begin{thm}
Any smooth toric Fano $3-$fold of Type I,II or III has a strongly exceptional sequence made up of line
bundles.
\end{thm}
In a more recent paper (\cite[Proposition 2.5]{frobenius}) they also proved, using the same tecniques we will
use in this article, this theorem below:
\begin{thm}
 The smooth toric Fano $3-$ fold of Type V has a strongly exceptional sequence made up of line bundles.
\end{thm}

\section{Bondal's Method and Thomsen's Algorithm}
We want to explain a method due to Bondal that we will use in the sequel. For that we recall briefly some
definitions and facts.
\begin{defin} 
Let $X$ be a smooth projective variety defined over an algebraically closed field $K$ of characteristic $0$.
\begin{itemize}
\item A coherent sheaf $\F$ on $X$ is \emph{exceptional} if $\Hom_{\ox}(\F,\F)=K$ and\linebreak
$\Ext^i_{\ox}(\F,\F)=0$
for $i>0$;
\item An ordered collection $(\F_0, \F_1,...,\F_m)$ of coherent sheaves on $X$ is an \emph{exceptional}
collection if each sheaf $\F_i$ is exceptional and\linebreak $\Ext^i_{\ox}(\F_k,\F_j)=0$ for $j<k$ and $i
\geq0$;
\item An exceptional collection $(\F_0,\F_1,...,\F_m)$ is a \emph{strongly exceptional} collection if in
addition $\Ext^i_{\ox}(\F_j,\F_k)=0$ for
$i\geq1$ and $j\leq k$;
\item An ordered collection $(\F_0,\F_1,...,\F_m)$ of coherent sheaves on $X$ is a \emph{full} (strongly)
exceptional collection if it is a (strongly) exceptional collection and $\F_0,\F_1,...,\F_m$ generate the
bounded derived category $D^b(X)$.
\end{itemize}
\end{defin}
\begin{rem}
The existence of a full strongly exceptional collection\linebreak $(\F_0,\F_1,...,\F_m)$ of coherent sheaves
on a smooth projective variety X implies
that the Grothendieck group $K_0(X )=K_0(\ox-mod)$ is isomorphic to $\mathbb{Z}^{m+1}$.
\end{rem}

Given any smooth complete toric variety $X$, in \cite{bondal} Bondal described a method to produce a
collection of line bundles on $X$, which is expected to be full
strongly exceptional for certain classes of toric Fano varieties. In particular he stated that for all
the smooth toric Fano $3$-folds but two these sequences were, indeed, strongly exceptional, and he does not
say anything about the remaining cases.
Let
see now how this method works.
Let $X$ be a smooth complete toric variety of dimension $n$ and $T$ an $n-$dimensional torus acting on it. So
for any integer $l\in \mathbb{Z}$,
there is a well-defined toric morphism
$$
\pi_{l}:X\ra X
$$
which restricts, on the torus $T$, to the Frobenius map
$$
\pi_{l}:T\ra T, \;\; t\mapsto t^l.
$$

This map is the factorization map with respect to the action of the group of $l$ torsion of $T$. Let us fix a
prime integer $p\gg0$,
$(\pi_p)_*(\ox)^{\dual}$ is a vector bundle of rank $p^n$ which splits into a sum of line bundles:
$$
(\pi_p)_*(\ox)^{\dual}=\oplus_{\chi}\ox(D_{\chi}),
$$
where the sum is taken over the group of characters of the $p$-torsion subgroup of $T$ (see Theorem 1 and
Proposition 2 \cite{thomsen}).
Moreover,
$$
c_1((\pi_p)_*(\ox)^{\dual})=\ox(-\frac{p^{n-1}(p-1)}{2}K_X)
$$
where $K_X$ is the canonical divisor of $X$. Bondal \cite{bondal} proved that the direct summands of
$(\pi_p)_*(\ox)^{\dual}$ generate the derived
category $D^b(X)$ and Thomsen \cite{thomsen} described an algorithm for computing explicitly the decomposition
of it. Let us summarize it for the case
of a smooth complete toric variety $X$ of dimension $n$, Picard number $\rho$  rank of group of
Grothendieck $s$.

We consider $\{\sigma_1,...,\sigma_s\}$ the set of maximal cones of the fan $\Sigma$ associated to $X$ and we
denote by $v_{i_1},...,v_{i_n}$ the
generators of $\sigma_i$. For each index $1\leq i\leq s$, we denote by $A_i\in \gl_n(\Z)$ the matrix having as
the $j$-th row the coordinates of
$v_{i_j}$ expressed in the basis $e_1,...,e_n$ of $N$ and by $B_i = A^{-1}_i\in\gl_n(\Z)$ its inverse. We
indicate with $w_{ij}$ the $j$-th column
vector in $B_i$. Introducing the symbols $X^{\hat{e^1}},...,X^{\hat{e^n}}$, we form the coordinate ring of the
torus $T\subset X$: 
$$
R=K[(X^{\hat{e^1}})^{\pm 1},...,(X^{\hat{e^n}})^{\pm 1}].
$$
Moreover the coordinate ring of the open affine subvariety $U_{\sigma_i}$ of $X$ corresponding to the cone
$\sigma_i$ is the subring
$$
R_i=K[X^{w_{i1}},...,X^{w_{in}}]\subset R,
$$
where $X^w:=(X^{\hat{e^1}})^{w_1}\cdots (X^{\hat{e^n}})^{w_n}$, if $w=(w_1,...,w_n)$, for simplicity we will
usually write $X_{ij}:=X^{w_{ij}}$.
For each $i$ and $j$, we denote by $R_{ij}$ the coordinate ring of $\sigma_i\cap \sigma_j$ and we define
$I_{ij}:=\{v\in M_{n\times1}(\Z):X^v_i \text{
is a unit in } R_{ij}\}$ and $C_{ij}:=B^{-1}_j B_i\in \gl_n(\Z)$, where we use the notation
$X^v_i:=(X_{i1})^{v_1}\cdots (X_{in})^{v_n}$ being $v$ a
column vector with entries $v_1,...,v_n$.

For every $p\in \N$ and $w\in I_{ij}$ , we define 
$$
P_p:=\{v \in M_{n\time1}(\Z):0\leq v_i <p\}
$$
and the maps
$$
h^w_{ijp}:P_p\ra R_{ij}, \\
r^w_{ijp} : P_p\ra P_p,
$$
by means of the following equality: 
$$
C_{ij}v+ w=p\cdot h^w_{ijp}(v)+r^w_{ijp}(v),\;\; \text{for any}\;\; v\in P_p.
$$
By \cite[Lemma 2 and Lemma 3]{thomsen}, these maps exist and they are unique.

Any toric Cartier divisor $D$ on $X$ can be represented in the form $\{(U_{\sigma_i},X^{u_i}_i)\}_{\sigma_i\in
\Sigma}$, $u_i\in M_{n\times1}(\Z)$
(see \cite[Chapter 3.3]{fulton}). So if we fix one of these representants, we can define
$u_{ij}=u_j-C_{ij}u_i$. If $\ox(D)=\ox$ is the trivial line
bundle, then for any pair $i,j$, we have $u_{ij}=0$.
For any $p\in \Z$ and any toric Cartier divisor $D$ on $X$ , we fix a set
$\{(U_{\sigma_i},X^{u_i}_i)\}_{\sigma_i\in \Sigma}$ representing $D$ and we
choose an index $l$ of a cone $\sigma_l\in \Sigma$. Let $D_v$, $v\in P_p$, denote the Cartier divisor
represented by the set
$\{(U_{\sigma_i},X^{h_i}_i)\}_{\sigma_i\in \Sigma}$ where, by definition $h_i = h^v_i:= h^{u_{li}}_{
lip}(v)$.

Then, we have
$$
(\pi_p)_*(\ox(D))^{\dual} =\bigoplus_{v\in P_p}\ox(D_v).
$$

If $h_i =(h_{i_1},...,h_{i_n})$ and $\alpha^j_{i1},...,\alpha^j_{in}$ are the entries of the $j$-th column
vector of $B_i$, then by definition:
$$
X^{h_i}_i=(X^{\hat{e_1}^{\alpha^1_{i1}}}\cdots
X^{\hat{e_n}^{\alpha^1_{in}}})^{h_{i1}}(X^{\hat{e_1}^{\alpha^2_{i1}}}\cdots
X^{\hat{e_n}^{\alpha^2_{in}}})^{h_{i2}}\cdots (X^{\hat{e_1}^{\alpha^n_{i1}}}\cdots
X^{\hat{e_n}^{\alpha^n_{in}}})^{h_{in}},
$$
and we indicate with:
$$
l_{\sigma_i}:=B_i\cdot
h_i=(\alpha^1_{i1}h_{i1}+\alpha^2_{i1}h_{i2}+...+\alpha^n_{i1}h_{in})\hat{e_1}+(\alpha^1_{i2}h_{i1}+\alpha^2_{
i2}h_{i2}+...+\alpha^n_{i2}h_{in})\hat{
e_2}+...
$$
$$
...+(\alpha^1_{in}h_{i1}+\alpha^2_{in}h_{i2}+...+\alpha^n_{in}h_{in})\hat{e_n}\in M=N^{\vee}.
$$
In this notation, if $D_v$ is the Cartier divisor represented by the set $\{(U_{\sigma_i},X^{h_i}_i)\}$, then 
$$
D_v=\beta^1_v Z_1+...+\beta^{n+\rho}_v Z_{n+\rho}
$$
where $\beta^j_v=-l_{\sigma_k}(v_j)$, for any maximal cone $\sigma_k$ containing the ray generator $v_j$
associated to the toric divisor $Z_j$. We can
observe that for any pair of maximal cones $\sigma_k$ and $\sigma_m$ containing $v_j$,  we have
$l_{\sigma_k}(v_j)=l_{\sigma_m}(v_j)$. 

We will now construct a full strongly exceptional collection for all toric Fano $3-$folds not covered in
\cite{tilting} and \cite{frobenius}. We will
analyze case by case.

\begin{nota}
 From now on, when it will be clear wich variety we will be working with, we shall omit the subscrit when
denoting a line bundle associated to a given divisor $D$: hence we will write $\fas{O}(D)$ instead of
$\fas{O}_X (D)$.
\end{nota}


\subsection{$\mathcal{D}_1=Bl_{\bbP^1}(\bbP_{\bbP^2}(\sO \oplus \sO(1)))$}
\begin{prop}
 Let $p$ be a sufficiently large prime integer and indicate with\linebreak
$\pi_p:\mathcal{D}_1\longrightarrow\mathcal{D}_1$ the Frobenius morphism relative to
$p$. Then
the different summands of $(\pi_p)_*(\fas{O}_{\mathcal{D}_1})^\vee$ are the following:
$$\fas{O},\;\fas{O}(Z_4+Z_5),\;\fas{O}(2Z_4+2Z_5),\;\fas{O}(Z_6),\;\fas{O}(Z_5+Z_6),\;\fas{O}(Z_4+Z_5+Z_6),
\;$$
$$
\fas{O}(Z_4+2Z_5+Z_6),\;\fas{O}(2Z_4+2Z_5+Z_6),\;\mathcal{O}(Z_6-Z_4).
$$
\end{prop}
\begin{proof}
According to \cite[Proposition 2.5.6]{batyrev}, the primitive collections of $\mathcal{D}_1$ are the
following:
$$\{v_3,v_6\},\;\{v_4,v_6\},\;\{v_3,v_5\},\; \{v_1,v_2,v_4\}, \{v_1,v_2,v_5\}.$$
On the other side,  the primitive relations  are:
\begin{itemize}
\item $v_3+v_6=0$;
\item $v_4+v_6=v_5$;
\item $v_3 +v_5=v_4$;
\item $v_1+v_2+v_4=2v_3$;
\item $v_1+v_2+v_5=v_3$.
\end{itemize}
We take the following three maximal cones:
$$ \sigma_1=\{v_1,v_2,v_3\},\;\sigma_2=\{v_1,v_2,v_6\}$$
$$\sigma_3=\{v_1,v_4,v_5\}$$
and, as a basis of $\mathbb{Z}^3$: $e_1=v_1$, $e_2=v_2$ $e_3=v_3$. Let $\widehat{e}_i$, for $i=1,2,3$, be the
dual basis. First of all we find the
coordinates of $v_4$, $v_5$ and $v_6$ in the system we have taken and we obtain:
$$v_4=(-1,-1,2),\;v_5=(-1,-1,1),\;v_6=(0,0,-1)$$
and thus we get the three matrices:
$$A_1=\left(
\begin{array}{ccc}
1 & 0 & 0\\
0 & 1 & 0\\
0 & 0 & 1
      \end{array}\right)
,\;A_2=\left(\begin{array}{ccc}
1 & 0 & 0\\
0 & 1 & 0\\
0 & 0 & -1
      \end{array}\right),A_3=\left(\begin{array}{ccc}
1 & 0 & 0\\
-1 & -1 & 2\\
-1 & -1 & 1
      \end{array}\right).$$
Inverting these matrices we obtain:
$$B_1=\left(
\begin{array}{ccc}
1 & 0 & 0\\
0 & 1 & 0\\
0 & 0 & 1
      \end{array}\right)
,\;B_2=\left(\begin{array}{ccc}
1 & 0 & 0\\
0 & 1 & 0\\
0 & 0 & -1
      \end{array}\right),B_3=\left(\begin{array}{ccc}
1 & 0 & 0\\
-1 & 1 & -2\\
0 & 1 & -1
      \end{array}\right).$$
Let us now take a vector $v=(a_1,a_2,a_3)^t$ in $P_p$ and let\linebreak $d_i=C_{1i}v=(B_i)^{-1}B_1v=A_iv$ for
$i=1,2,3$. Then we have 
$$d_2=(a_1,a_2,-a_3),\;\text{and}\;d_3=(a_1,-a_1-a_2+2a_3,-a_1-a_2+a_3).$$
Recall that $(\pi_p)_{*}(\ox)^\vee=\bigoplus_{v\in P_p}\ox(D_v)$, with
\begin{equation}\label{Dv}
 D_v=-l_{\sigma_3}(v_4)Z_4-\ellesigma{3}{5}Z_5-\ellesigma{2}{6}Z_6.
\end{equation}
Let us see how $D_v$ varies when we took different $v$'s in $P_p$.
\\

\indent \underline{Case 1:} $v=0$. In this case $D_v=0$.\\

\indent \underline{Case 2:} $a_1\neq0$, $a_2=a_3=0$.\\
Computing $d_2$ and $d_3$ we find that $d_2=(a_1,0,0)$, while $d_3=(a_1,-a_1,-a_1)$. Thus
$$d_2=p\cdot (0,0,0)+(a_1,0,0)\;\text{and}\; d_3=p\cdot(0,-1,-1)+(a_1,p-a_1,p-a_1)$$
 and consequently, $h_2=0$ and $h_3=(0,-1,-1)$. It follows that $\elsig{2}$ is the zero operator and
$\elsig{3}=\widehat{e}_2$. Using the formulas
\eqref{Dv} we gain:
$$ D_v=-\ellesigma{3}{4}Z_4-\ellesigma{3}{5}Z_5=Z_4+Z_5.$$
\\

\indent \underline{Case 3:} $a_2\neq0$, $a_1=a_3=0$.\\
In this case we have $d_2=(0,a_2,0)$ and $d_3=(0,-a_2,-a_2)$. As in the previous situation we get
$$D_v=Z_4+Z_5.$$
\\

\indent \underline{Case 4:} $a_3\neq0$, $a_2=a_1=0$.\\
Under the aforementioned hypothesis, $d_2=(0,0,-a_3)$, thus $h_2=(0,0,-1)$. The triple $d_3$ will instead be
$(0,2a_3,a_3)$. So we have now two
possibilities for $h_3$: $h_3=(0,s,0)$ with $s\in\{0,1\}$ assuming the following values:
$$s=\begin{cases}
     0&\text{if $2a_3< p$,}\\
1&\text{if $2a_3\geq p$.}
    \end{cases}
$$
The next step is to compute the operators $\elsig{2}$ and $\elsig{3}$. We have $\elsig{2}=\widehat{e}_3$ and
$\elsig{3}=s\widehat{e}_2+s\widehat{e}_3$.
Concluding we got
$$
D_v=-\ellesigma{3}{4}Z_4-\ellesigma{3}{5}Z_5-\ellesigma{2}{6}Z_6=-sZ_4+Z_6=\begin{cases}
                                                                  Z_6,\\
Z_6-Z_4.
                                                                 \end{cases}
$$

\indent \underline{Case 5:} $a_1,\;a_2\neq0$, $a_3=0$.\\
In this case $d_2=(a_1,a_2,0)$, that means that $h_2=0$ which implies that $\elsig{2}$ is the zero operator.
On the other side
$d_3=(a_1,-a_1-a_2,-a_1-a_2)$ and, as in the previous case, we have two different possibilities for $h_3$:
$$h_3=\begin{cases}
       (0,-1,-1)& \text{If $-a_1-a_2\geq -(p-1)$},\\
       (0,-2,-2)&\text{If $-a_1-a_2< -(p-1)$}.
      \end{cases}$$
Thus we can write $h_3=(0,-s,-s)$ where $s=1,2$. With this notation we can easily compute the operator
$\elsig{3}=s\ei{2}$ and get
$$D_v=\begin{cases}
       Z_4+Z_5,\\
2Z_4+2Z_5.
      \end{cases}$$
\\

\indent \underline{Case 6:} $a_1,\;a_3\neq0$, $a_2=0$.\\
Under these assumptions, $d_2=(0,a_2,-a_3)$, $h_2=(0,0,-1)$ and, as we already calculated in the solution to
the fourth case, $\elsig{2}=\ei{3}$.
Since\linebreak $d_3=(a_1,-a_1+2a_3,-a_1+a_3)$ we have several possibilities for $h_3$, depending on the sign
of $-a_1+2a_3$ and,$-a_1+a_3$:\\
\textit{Case 6.1:} $-a_1+a_3\geq 0$.\\
If this is the case, then $h_3=0$ and $\elsig{3}$ is the zero operator. Thus we get
$$D_v=Z_6.$$
\textit{Case 6.2:} $-a_1+a_3<0$, $-a_1+2a_3\geq 0$.\\
Under these hypothesis $h_3=(0,0,-1)$ and $\elsig{3}=2\ei{2}+\ei{3}$. Consequently we get
$$D_v=Z_5+Z_6.$$
\textit{Case 6.3:} $-a_1+2a_3<0$.\\
In this case $h_3=(0,-1,-1)$, and $\elsig{3}=\ei{2}$. Making all the computation we find out that
$$D_v=Z_4+Z_5+Z_6.$$

\indent \underline{Case 7:} $a_3,\;a_2\neq0$, $a_1=0$.\\ 
After having switched $a_1$ with $a_2$, this case is completely identical to the previous one, and we get the
same divisors $D_v$'s.\\

\indent \underline{Case 8:} $a_1,\;a_2,\;a_3\neq0$.\\
As before, in this case we have $h_2=(0,0,-1)$ and $\ellesigma{2}{6}=-1$. Let us compute $h_3$ and
$\elsig{3}$.\\
\textit{Case 8.1:} $-a_1-a_2+a_3\geq0$.\\
In this case $h_3=0$ and $D_v=Z_6$.\\
\textit{Case 8.2:} $-(p-1)\leq-a_1-a_2+a_3<0$, $-a_1-a_2+2a_3\geq 0$.\\
In this case $h_3=(0,0,-1)$ and 
$$D_v=Z_5+Z_6.$$
\textit{Case 8.3:} $-(p-1)\leq-a_1-a_2+a_3\leq-a_1-a_2+2a_3<0$.\\
In this case $h_3=(0,-1,-1)$ and $\elsig{3}=\ei{2}$. It follows that
$$D_v=Z_4+Z_5+Z_6.$$
\textit{Case 8.4:} $-(p-1)>-a_1-a_2+a_3$, $0>-a_1-a_2+2a_3\geq -(p-1)$.\\
In this case $h_3=(0,-1,-2)$. Consequently $\elsig{3}=3\ei{2}+\ei{3}$. Thus we have
$$D_v=Z_4+2Z_5+Z_6.$$
\textit{Case 8.5:}  $-a_1-a_2+2a_3< -(p-1)$.\\
Under this assumption $h_3=(0,-2,-2)$ and $\elsig{3}=2\ei{2}$. As a consequence we can compute
$$D_v=2Z_4+2Z_5+Z_6.$$

Putting together what we calculated in each case we obtain the statement.

\end{proof}

You may observe that the output of Thomsen's algorithm consists of nine different line bundles. Since
$\mathrm{rk}( K_0(\duno))=8$, there is no hope
for
this sequence to be full strongly exceptional. Thus we have first to find an eight items long full sequence
and, afterward, prove that it is strongly
exceptional. We accomplish the first of these two goals by proving the following proposition.
\begin{prop}\label{full d1}
The bounded derived category $D^b(\duno)$ is generated by the following line bundles:
$$\fas{O},\;\fas{O}(Z_4+Z_5),\;\fas{O}(2Z_4+2Z_5),\;\fas{O}(Z_6),\;\fas{O}(Z_5+Z_6),\;\fas{O}(Z_4+Z_5+Z_6),
\;$$
$$
\fas{O}(Z_4+2Z_5+Z_6),\;\fas{O}(2Z_4+2Z_5+Z_6).
$$
\end{prop}
\begin{proof}
 It is enough to prove that $\mathcal{O}(Z_6-Z_4)$ is in the triangulated category generated by all the other
line bundles. We do that by constructing
an exact sequence in which $\mathcal{O}(Z_6-Z_4)$ appears once, and all other objects in the sequence are
direct sums of aforementioned line bundles.
Let us consider the primitive collection $\{v_1,v_2,v_4\}$ and construct the Koszul complex associated to
$\mathcal{O}(-Z_1)\oplus\mathcal{O}(-Z_2)\oplus\mathcal{O}(-Z_4)$:
\begin{multline}\label{koszul}
\quad
\quad\quad0\leftarrow\mathcal{O}_Y\leftarrow\mathcal{O}\leftarrow\mathcal{O}(-Z_1)\oplus\mathcal{O}
(-Z_2)\oplus\mathcal{O}(-Z_4)\leftarrow\\
\mathcal{O}(-Z_1-Z_2)\oplus\mathcal{O}(-Z_2-Z_4)\oplus\mathcal{O}(-Z_1-Z_4)\leftarrow\mathcal{O}
(-Z_1-Z_2-Z_4)\leftarrow 0
\end{multline}
where $Y$ denotes the intersection $Z_1\cap Z_2\cap Z_4$. Since we started out with a primitive collection,
then $Y=\emptyset$ and \eqref{koszul} is
exact.\\
Now we write all the divisors in the basis of $\mathrm{Pic}(X)$ given by $Z_4$, $Z_5$ and $Z_6$. Afterward we
dualize the Koszul complex and twist it
by $\mathcal{O}(Z_6-Z_4)$. We obtain the following exact sequence:
\begin{multline}\label{koszul2}
\quad\quad\quad0\rightarrow\mathcal{O}(Z_6-Z_4)\rightarrow\mathcal{O}(Z_6+Z_5)^2\oplus\mathcal{O}
(Z_6)\rightarrow\cdots\\
\cdots\rightarrow\mathcal{O}(Z_4+Z_5+Z_6)^2\oplus\mathcal{O}(Z_4+2Z_5+Z_6)\rightarrow\mathcal{O}
(2Z_4+2Z_5+Z_6)\rightarrow 0.\end{multline}
Observe that all the line bundles in \eqref{koszul2} but the first one are among those we picked as
generators, hence the statement is proved.
\end{proof}

\subsection{$\ddue=Bl_{\bbP^1}(\bbP^2\times \bbP^1)$}
In this paragraph we are concerned with finding a set of generators for $D^b(\ddue)$ of cardinality 8, that is
of cardinality the rank of the
Grothendieck group $K_0(\ddue)$.
\begin{prop}\label{full d2}
 Let $p$ be a sufficiently large prime integer and indicate with\linebreak $\pi_p:\ddue\longrightarrow\ddue$
the Frobenius morphism relative to $p$. Then the
different
summands of $(\pi_p)_*(\fas{O}_{\mathcal{D}_2})^\vee$ are the following:
$$\mathcal{O},\;\mathcal{O}(Z_4+Z_5),\;\mathcal{O}(2Z_4+2Z_5),\;\mathcal{O}(Z_6),\;\mathcal{O}(Z_5+Z_6),
\;\mathcal{O}(Z_4+Z_5+Z_6),$$
$$
\;\mathcal{O}(Z_4+2Z_5+Z_6),\;\mathcal{O}(2Z_4+2Z_5+Z_6).
$$
\end{prop}
\begin{proof}
 From the classification of toric Fano 3-folds (\cite[Proposition 2.5.6]{batyrev}) we know that the  primitive
collections of $\mathcal{D}_2$ are
the same of $\duno$, while primitive relations are:
\begin{itemize}
\item $v_3+v_6=0$;
\item $v_4+v_6=v_5$;
\item $v_3 +v_5=v_4$;
\item$v_1+v_2+v_4=v_3$;
\item $v_1+v_2+v_5=0$.
\end{itemize}
As we already did, choose $(v_1,v_2,v_3)$ as a basis of $\mathbb{Z}^3$, and take the following three maximal
cones that covers all the
vertices of the polytope defining the toric variety:
$$\sigma_1=\{v_1,v_2,v_3\},\quad\sigma_2=\{v_1,v_2,v_6\},$$
$$\sigma_3=\{v_1,v_5,v_4\}.$$
One can easily see that the matrices $A_i$ but the last one unchanged respect to the ones we processed in the
case of $\duno$, while 
$$ A_3=\left(\begin{array}{ccc}
1 & 0 & 0\\
-1 & -1 & 0\\
-1 & -1 & 1
      \end{array}\right).$$
We need to compute its inverse:
$$B_3=\left(\begin{array}{ccc}
1 & 0 & 0\\
-1 & -1 & 0\\
0 & -1 & 1
      \end{array}\right).$$
Now we take $v=(a_1,a_2,a_3)^t\in P_p$ and process it obtaining\linebreak $d_3=A_3v=(a_1,a_1-a_2,a_2-a_3)^t$.
Let us see how $D_v$ changes if we take
different $v$'s.
It
is useful to observe that, since $A_1$ and $A_2$ are unchanged, we do not need to be concerned with the
computation of $\elsig{1}$ and $\elsig{2}$
because we can obtain them from the previous example.
Almost the same computations we did for the $\duno$ case show that
\begin{enumerate}
 \item If all the $a_i$'s are null then $D_v=0$.
\item If $a_1\neq0$, while $a_2=a_3=0$, then $D_v=Z_4+Z_5$.
\item If $a_2\neq0$, while $a_1=a_3=0$, then $D_v=Z_4+Z_5$.
\item If $a_3\neq0$, while $a_2=a_1=0$, then $D_v= Z_6$.
\item If $a_1, a_2\neq0$, while $a_3=0$, the coordinate of $\elsig{3}$ in the dual basis are $(0,s,0)$ with
$s=0,1$. Thus we have two possibilities 
for $D_v$:
$$D_v=\begin{cases}
       Z_4+Z_5,\\
2Z_4+2Z_5.
      \end{cases}$$
\item If $a_1, a_3\neq0$, while $a_2=0$, then again we have two possibility for $D_v$, since
$\elsig{3}=(0,1,1-s)$, with $s=0,1$ in the
dual basis. We obtain
$$D_v=\begin{cases}
       Z_5+Z_6,\\
Z_4+Z_5+Z_6.
      \end{cases}$$
\item If $a_3, a_2\neq0$, while $a_1=0$, then $\elsig{3}$ is again $(0,1,1-s)$ with $s=0,1$. Thus we get the
same divisor we obtained in the previous case.
\item Finally, if all the $a_i$'s are not zero, we have to split our computation in four sub-cases, depending
on the values $-a_1-a_2$ and
$-a_1-a_2+a_3$ will assume. We end up with four possibilities for $D_v$:
$$D_v=\begin{cases}
       2Z_4+2Z_5+Z_6,\\
Z_4+2Z_5+Z_6,\\
Z_4+Z_5+Z_6,\\
Z_5+Z_6.
      \end{cases}$$
\end{enumerate}
The union of all these intermediate results implies the statement.
\end{proof}


\subsection{$\mathcal{E}_1$, $S_2-$bundle over $\bbP^1$ and $\mathcal{E}_2=S_2-$bundle over $\bbP^1$}
\begin{prop}
 Let $p$ be a sufficiently large prime and indicate by\linebreak $\pi_p:\euno\longrightarrow\euno$ the
Frobenius morphism relative to $p$. Then the distinct summands of
$(\pi_p)_*(\fas{O}_{\mathcal{E}_1})^\vee$ are the following:
$$\mathcal{O},\;\mathcal{O}(Z_7),\;\mathcal{O}(Z_4),\;\mathcal{O}(Z_4+Z_7),\;\mathcal{O}(Z_4+Z_5),\;\mathcal{O
}(Z_4+Z_5+Z_7),$$
$$
\mathcal{O}(Z_1+Z_5+2Z_7),\;\mathcal{O}(Z_1+Z_4+Z_5+Z_7),\;\mathcal{O}(Z_1+Z_4+Z_5+2Z_7).
$$
\end{prop}
\begin{proof}
 From \cite[Proposition 2.5.9]{batyrev} we know that the primitive collections of $\mathcal{E}_1$ are:
$$
\{v_2,v_4\},\;\{v_3,v_5\},\;\{v_1,v_3\},\;\{v_2,v_5\},\;\{v_1,v_4\},\;\{v_6,v_7\},
$$
and the ray generators of $\mathcal{E}_1$ satisfy the following relations:
\begin{itemize}\item$v_2+v_4=0$,
\item$v_3+v_5=0$,
\item$v_1+v_3=v_2$,
\item$v_2+v_5=v_1$,
\item$v_1+v_4=v_5$,
\item$v_6+v_7=v_1$.
\end{itemize}
As in the previous examples we choose three maximal cones such that they cover all the ray generators of the
toric Fano $3-$fold:
$$\sigma_1=\{v_2,v_3,v_6\},\quad\sigma_2=\{v_4,v_5,v_6\},$$
$$\sigma_3=\{v_2,v_1,v_7\}.$$
Let us take $\mathfrak{B}=(v_2=e_1,v_3=e_2,v_6=e_3)$ a basis of $\mathbb{Z}^3$ and denote with $\widehat{e}_i$
the elements of the dual basis. For
$i=1,2,3$ we construct the matrices $A_i$ and we get: 
$$A_1=\left(
\begin{array}{ccc}
1 & 0 & 0\\
0 & 1 & 0\\
0 & 0 & 1
      \end{array}\right)
,\;A_2=\left(\begin{array}{ccc}
-1 & 0 & 0\\
0 & -1 & 0\\
0 & 0 & 1
      \end{array}\right),A_3=\left(\begin{array}{ccc}
1 & 0 & 0\\
1 & -1 & 0\\
1 & -1 & -1
      \end{array}\right).$$
Inverting these matrices we obtain:
$$B_1=\left(
\begin{array}{ccc}
1 & 0 & 0\\
0 & 1 & 0\\
0 & 0 & 1
      \end{array}\right)
,\;B_2=\left(\begin{array}{ccc}
-1 & 0 & 0\\
0 & -1 & 0\\
0 & 0 & 1
      \end{array}\right),B_3=\left(\begin{array}{ccc}
1 & 0 & 0\\
1 & -1 & 0\\
0 & 1 & -1
      \end{array}\right).$$
Now choose $v=(a_1,a_2,a_3)^t\in P_p$ and let
$$d_1=A_1v=(a_1,a_2,a_3),\quad d_2=A_2v=(-a_1,-a_2,a_3),$$
$$d_3=A_3v=(a_1,a_1-a_2,a_1-a_2-a_3).$$
Defining as we did before $h_i$ and $\elsig{i}$, then we know that
\begin{gather*}
 D_v=-l_{\sigma_3}(v_1)Z_1-l_{\sigma_1}(v_2)Z_2-l_{\sigma_1}(v_3)Z_3-l_{\sigma_2}(v_4)Z_4+\\
-l_{\sigma_2}(v_5)Z_5-l_{\sigma_1}(v_6)Z_6-l_{\sigma_3}(v_7)Z_7.
\end{gather*}
Note that $A_1v\in P_p$ for every $v$. So $h_1=(0,0,0)$ for every $v$ and $l_{\sigma_1}$ is the zero operator.
Thus we can write:
$$
 D_v=-l_{\sigma_3}(v_1)Z_1-l_{\sigma_2}(v_4)Z_4-l_{\sigma_2}(v_5)Z_5-l_{\sigma_3}(v_7)Z_7.$$
To find all the $D_v$'s we have to see how the $h_i$'s change when $v$ varies in $P_p$.\\

\indent \underline{Case 1:} $a_1=a_2=a_3=0$.
As before in this case we have $D_v=0$.\\

\indent \underline{Case 2:} $a_1\neq0$, $a_2=a_3=0$.\\
In this case $d_2=(-a_1,0,0)$ and $d_3=(a_1,a_1,a_1)$. So we obtain that $\elsig{3}$ is the zero operator,
while $\elsig{2}=\ei{1}$. Putting this
together
and computing the coefficients $\ellesigma{i}{j}$  we get that
$$D_v=Z_4.$$

\indent \underline{Case 3:} $a_2\neq0$, $a_1=a_3=0$.\\
In this case we compute $h_2$ to be the vector $(0,-1,0)$, while $h_3=(0,-1,-1)$. As a consequence
$\elsig{2}=\ei{2}=\elsig{3}$. Then
$$D_v=Z_1+Z_5+Z_7.$$

\indent \underline{Case 4:} $a_3\neq0$, $a_1=a_2=0$.\\
In this case $\elsig{2}$ is the  zero operator. Computing $\elsig{3}$ we get instead that\linebreak
$h_3=(0,0,-1)$, $\ellesigma{3}{1}=0$ and
$\ellesigma{3}{7}=-1$. It follows that
$$D_v=Z_7.$$

\indent \underline{Case 5:} $a_1,\;a_2,\neq0$, $a_3=0$.\\
In this case $h_2$ is always equal to $(-1,-1,0)$ while we have two hypothesis for $h_3$. Indeed, being
$d_3=(a_1,a_1-a_2,a_1-a_2)$, we have that
$h_3=(0,-s,-s)$ with 
$$
s=\begin{cases}
   0 &\text{if $a_1-a_2\geq 0$},\\
1&\text{otherwise.}
  \end{cases}
$$
It follows that $\elsig{2}=\ei{1}+\ei{2}$ and $\elsig{3}=s\ei{2}$. Consequently we get
$$
D_v=\begin{cases}
     Z_4+Z_5,\\
Z_1+Z_4+Z_5+Z_7.
    \end{cases}
$$

\indent \underline{Case 6:} $a_1,\;a_3\neq0$, $a_2=0$.\\
In this case we have $h_2=(-1,0,0)$ and, again, we get two possibilities for $h_3$. More precisely
$$h_3=\begin{cases}
       0&\text{if $a_1-a_3\geq0$},\\
(0,0,-1)&\text{otherwise.}
      \end{cases}
$$
As a consequence in this case we obtain two different summands of $(\pi_p)_*(\fas{O}_{\fas{E}_1})^{\vee}$:
$$D_v=\begin{cases}
      Z_4,\\
Z_4+Z_7.
     \end{cases}
$$
\indent \underline{Case 7:} $a_3,\;a_2\neq0$, $a_1=0$.\\
In this case $h_2=(0,-1,0)$. Thus $\ellesigma{2}{4}=0$ and $\ellesigma{2}{5}=-1$. It remains to compute
$\elsig{3}$. Being $d_3=(0,-a_2,-a_2-a_3)$ we
have that $h_3=(0,-1,-s)$ with $s=1$ or $s=2$. It follows that $\elsig{3}=\ei{2}+(s-1)\ei{3}$ and we end up
with the following two divisors:
$$D_v=\begin{cases}
  Z_1+Z_5+Z_7,\\
Z_1+Z_5+2Z_7.     
      \end{cases}
$$

\indent \underline{Case 8:} $a_1,\;a_2,\;a_3\neq0$.\\
In this case $h_2=(-1,-1,0)$ and $l_{\sigma_2}(v_4)=l_{\sigma_2}(v_5)=-1$.
To compute $h_3$ we need to consider some different situations.\\
\textit{Case 8.1:} $a_1-a_2-a_3<-(p-1)$.\\
Under this assumption $a_1-a_2<0$ and hence $h_3=(0,-1,-2)$. It follows that $l_{\sigma_3}(v_1)=-1$, while
$l_{\sigma_3}(v_7)=-2$. Putting this together
with what was found earlier we get
$$D_v=Z_1+Z_4+Z_5+2Z_7.$$
\textit{Case 8.2:} $0>a_1-a_2-a_3\geq-(p-1)$, $a_1-a_2<0$.\\
In this case $h_3=(0,-1,-1)$ and $l_{\sigma_3}(v_1)=l_{\sigma_3}(v_7)=-1$. Then
$$D_v=Z_1+Z_4+Z_5+Z_7.$$
\textit{Case 8.3:} $0>a_1-a_2-a_3\geq-(p-1)$, $a_1-a_2\geq0$.\\
We have that $h_3=(0,0,-1)$, $l_{\sigma_3}(v_1)=0$ and $l_{\sigma_3}(v_7)=-1$. So
$$
D_v=Z_4+Z_5+Z_7.
$$
\end{proof}

Using the same methods and techniques, we were also able to prove the following proposition. Since the steps
of
the proof and all the calculations are pretty much the same as the ones in the previous case, we decided, for
the sake of brevity, to ommit them and to leave them to the reader.
\begin{prop}
 Let $p$ a sufficiently large prime and indicate with\linebreak $\pi_p:\edue\longrightarrow\edue$ the
Frobenius morphism relative to $p$. Then the
different summands of $(\pi_p)_*(\fas{O}_{\mathcal{E}_2})^\vee$ are the following:
$$\mathcal{O},\;\mathcal{O}(Z_7),\;\mathcal{O}(Z_4),\;\mathcal{O}(Z_1+Z_5),\;\mathcal{O}(Z_4+Z_7),\;\mathcal{O
}(Z_4+Z_5),\;\mathcal{O}(Z_1+Z_5+Z_7),
$$
$$
\mathcal{O}(Z_4+Z_5+Z_7),\;\mathcal{O}(Z_1+Z_4+Z_5),\;\mathcal{O}(Z_1+Z_4+Z_5+Z_7).
$$
\end{prop}

\subsection{$\mathcal{E}_4=S_2-$bundle over $\bbP^1$}

\begin{prop}
 Let $p$ be a sufficiently large prime and indicate with\linebreak $\pi_p:\equat\longrightarrow\equat$ the
Frobenius morphism relative to $p$. Then the
different summands of $(\pi_p)_*(\fas{O}_{\mathcal{E}_4})^\vee$ are the following:
$$\mathcal{O},\;\mathcal{O}(Z_7),\;\mathcal{O}(Z_4),\;\mathcal{O}(Z_1+Z_5),\;\mathcal{O}(Z_4+Z_7),\;\mathcal{O
}(Z_4+Z_5),\;\mathcal{O}(Z_1+Z_5+Z_7),
$$
$$
\mathcal{O}(Z_4+Z_5+Z_7),\;\mathcal{O}(Z_1+Z_4+Z_5),\;\mathcal{O}(Z_1+Z_4+Z_5+Z_7).
$$
\end{prop}
\begin{proof}
 The primitive collections of $\mathcal{E}_4$ are the same of $\euno$ while the primitive relations are the
following:
\begin{itemize}
\item $v_2+v_4=0$,
\item$v_3+v_5=0$,
\item$v_1+v_3=v_2$,
\item$v_2+v_5=v_1$,
\item$v_1+v_4=v_5$,
\item$v_6+v_7=v_3$.
\end{itemize}
As before we chose $(v_2,v_3,v_6)$ as a basis of $\mathbb{Z}^3$, and we take the following three maximal cones
that covers all the
vertices of the polytope defining the toric variety, namely:
$$\sigma_1=\{v_2,v_3,v_6\},\quad\sigma_2=\{v_4,v_5,v_6\},$$
$$\sigma_3=\{v_2,v_1,v_7\}.$$
Since only one primitive relation has changed, all the matrices $A_i$ but the last one are the same as in the
previous two cases, while 
$$ A_3=\left(\begin{array}{ccc}
1 & 0 & 0\\
1 & -1 & 0\\
1 & 0 & -1
      \end{array}\right).$$
Because of this our only concern is to calculate $\elsig{3}$. In order to do so we just need to calculate the
inverse of $A_3$:
$$B_3=\left(\begin{array}{ccc}
1 & 0 & 0\\
1 & -1 & 0\\
1 & 0 & -1
      \end{array}\right).$$
As usual we take $v=(a_1,a_2,a_3)^t\in P_p$ and we compute\linebreak $d_3=(a_1,a_1-a_2,a_2-a_3)$. Let us see
how $D_v$ changes if we take different
$v$'s.
Almost the same computations we did in the previous cases show that:
\begin{enumerate}
 \item If all the $a_i$'s are null then $D_v=0$.
\item If $a_1\neq0$, while $a_2=a_3=0$, then $D_v=Z_4$.
\item If $a_2\neq0$, while $a_1=a_3=0$, then $D_v=Z_1+Z_5$.
\item If $a_3\neq0$, while $a_2=a_1=0$, then $D_v= Z_7$.
\item If $a_1, a_2\neq0$, while $a_3=0$, the coordinates of $\elsig{3}$ in the dual basis are $(0,s,s)$ with
$s=0,1$. Thus we have two possibilities
for $D_v$:
$$D_v=\begin{cases}
       Z_4+Z_5,\\
Z_1+Z_4+Z_5.
      \end{cases}$$
\item If $a_1, a_3\neq0$, while $a_2=0$, then $D_v=Z_4+Z_7$.
\item If $a_3, a_2\neq0$, while $a_1=0$, then $\elsig{3}=(0,1,1+s)$ with $s=0,1$. It follows that 
$$D_v=\begin{cases}
       Z_1+Z_5,\\
Z_1+Z_5+Z_7.
      \end{cases}$$
\item Finally, if all the $a_i$'s are not zero, we get four possibilities for $D_v$:
$$D_v=\begin{cases}
       Z_4+Z_5,\\
Z_1+Z_4+Z_5,\\
Z_4+Z_5+Z_7,\\
Z_1+Z_4+Z_5+Z_7.
      \end{cases}$$
\end{enumerate}
If we collect all cases we obtain exactly the divisors appearing in the statement above and hence the
proposition is proved.

\end{proof}



\section{Vanishing Theorems}
Given a full sequence of line bundles, $(\fas{L}_1,\ldots,\fas{L}_n)$, in order to check if it is also
strongly exceptional, we have to prove that:
\begin{itemize}
 \item[(i)] $\mathrm{Ext}^i_{\ox}(\fas{L}_\alpha,\fas{L}_\beta)=0$ for all $i>0$ and for all $\alpha$ and
$\beta$.
\item[(ii)] $\mathrm{Hom}(\fas{L}_\alpha,\fas{L}_\beta)=0$ for all $\alpha>\beta$.
\end{itemize}
We briefly recall the definition of acyclic line bundle:
\begin{defin}
 A line bundle $\mathcal{L}$ on a smooth complete toric variety $Y$ is said to be \emph{acyclic} if 
$$H^i(Y,\mathcal{L})=0\quad\text{for every $i\geq1$}.$$
\end{defin}
We can substitute conditions (i) and (ii) by other two equivalent statements:
\begin{itemize}
 \item[(i)'] $\fas{L}_\beta\otimes\inv{\fas{L}_\alpha}$ is acyclic for all $\alpha$ and $\beta$.
\item[(ii)'] $\fas{L}_\beta\otimes\inv{\fas{L}_\alpha}$ has no global section for all $\alpha>\beta$.
\end{itemize}
Verify (ii)' is quite simple: in fact let $Z_1 , \ldots , Z_m$ be the toric divisors on a toric variety $X$.
For every $a=(a_1,\ldots,a_m)\in
\mathbb{Z}^m$ we can consider the subset $I_a=\{i_1,\ldots,i_s\}\subseteq J=\{1,\ldots,m\}$  such that
$\rho\in I_a$ if and only if $a_\rho\geq 0$.
Now to $I_a$, as well as to every subset of $J$, we associate the simplicial subcomplex, $C_{I_a}$, of the fan
$\Sigma(X)$ which consists of the cones
in $\Sigma(X)$ whose rays lie in $I_a$. It is a known fact that
\begin{prop}
With the above notation, given $D=\sum_\rho a_\rho Z_\rho$ a toric divisor on $X$ then
$$H^p(X,\mathcal{O}(D))=\bigoplus_{a'}H_{\mathrm{dim}(X)-p}(C_{I_{a'}}, K)$$
where the sum is taken over all $a'=(a'_1,\ldots,a'_m)$ such that $D$ is lineraly equivalent to $\sum_\rho
a'_\rho Z_\rho$.
\end{prop}
\begin{proof}
\cite[Proposition 4.1]{borisov}
\end{proof}
A consequence of this result was given in \cite[Corollary 2.8]{toricfiber}:
\begin{cor}
With the above notation, $H^0(X, \mathcal{O}(D))$ is determined only by $a' = (a'_1,\ldots,a'_m)$ such that
$D$ is linearly equivalent to $\sum_\rho
a'_\rho Z_\rho$, $a'_\rho\geq 0$. We call the toric divisors such as those \emph{toric effetive} divisors.
\end{cor}
Thus to verify that $\fas{L}$ satisfies (ii)' write $\fas{L}_\beta\otimes\inv{\fas{L}_\alpha}$ as $\fas{O}(D)$
with $D$ a $T$-Cartier divisor on the
toric variety. Then it is enough to prove that it  $D$ is not linearly equivalent to a linear combination with
non-negative coefficients of the
principal toric divisors.\\
\indent Condition (i)' is more laborious. In order to verify it we need some acyclicity criteria. The first
one it is an easy consequence of a vanishing theorem by Mustata whose statement we briefly recall:
\begin{thm}[Mustata]
Let $Y$ be a complete smooth toric variety, and be $\fas{L}$ an ample line bundle on $Y$. If
$Z_{i_1},\;\ldots,\;Z_{i_k}$ are distinct toric divisors of $Y$, then the line bundle
$\fas{L}\otimes\fas{O}_X(-Z_{i_1}-\cdots-Z_{i_k})$ is acyclic.
\end{thm}
\begin{proof}
 \cite[Corollary 2.5]{must}
\end{proof}
\begin{rem}\label{mustata}
 If $X$ is a toric Fano variety, then all the  line bundles of the form
$$\fas{O}_X\left(\sum_i \varepsilon_iZ_i\right)$$
with $Z_i$ principal toric divisors and $\varepsilon_i\in\{0,1\}$ are acyclic.
\end{rem}

Observing the statement of condition (i)', it is evident that the hypothesis of Proposition \ref{mustata} are
too  strong to apply to all the line
bundles we need to be acyclic. Indeed, the set of divisors we have to check is ``symmetric'' in the sense that
if $D$ is one element of the set, than
also $-D$ is an element. But (toric) effectiveness is not a symmetric propriety, thus if we can use the
previous criterion to check the acyclicity of
a line bundle $\fas{L}$, surely we will not be able to apply it to its dual. In conclusion we need a weaker
criterion. In \cite{borisov}, Borisov and
Hua gave sufficient and necessary conditions for a line bundle on a toric variety to be acyclcic. In the next
paragraphs we will explain their
method.\\
\indent Let us introduce the notion of forbidden set.
\begin{defin}
 Let $X$ a toric variety with $m$ ray  generators. For every proper subset $I\subsetneq J=\{1,\ldots,m\}$
consider the associated simplicial complex
$C_I$. We say that $I$ is a \emph{forbidden} set if $C_I$ has a non-trivial homology.
\end{defin}
\begin{exa}
 It is easy to see that $\{2,4,5\}$ is a forbidden set for $\mathcal{E}_1$. In fact the simplicial complex
$C_I$ consists in 0 maximal cones, one face
and 3 edges. Its (reduced) chain complex is
$$0\rightarrow K\rightarrow K^3\rightarrow K\rightarrow 0$$
that is obviously not exact. Then $C_I$ has a non-trivial homology.
\end{exa}

Thanks to the following result, forbidden sets play a key role in proving the acyclicity of line bundles on
toric Fano varieties.
\begin{prop}[Borisov-Hua] Let $X$ a toric Fano variety and consider all forbidden sets $I\subset J$. For
each of them consider the line bundles of the form
\begin{equation}\label{forbidden}
 \mathcal{O}_X\left(-\sum_{i\notin I}Z_i+\sum_{i\in I}a_iZ_i-\sum_{i\notin I}a_iZ_i\right)
\end{equation}
with $a_i\in\mathbb{Z}_{\geq 0}$ for every $i\in J$. Then $\mathcal{L}\in\mathrm{Pic}(X)$ is acyclcic if and
only if $\mathcal{L}$ is not of the form
\eqref{forbidden}.
\end{prop}
\begin{proof}
 \cite[Proposition 4.3]{borisov}
\end{proof}

We call the line bundles like \eqref{forbidden} \emph{forbidden line bundles} (or \textit{forbidden forms})
relative to the set $I$. By abuse of
notation we shall say that a divisor $D$ is of a forbidden form relative to $I$ or that it can be put in a
forbidden form relative to $I$ if it
is linearly equivalent to a divisor $F$ such that $\fas{O}(F)$ is a forbidden form.\\

In what follows we will prove that the full sequences we got in the previous sections are indeed strongly
exceptional. Although Boris-Hua method's
gives sufficient and necessary condition to a line bundle to be acyclic, it needs a lot of tedious
calculations and so we will use also  the first
criterion to lessen the number of divisors we need to check.

\subsection{$\mathcal{D}_1$ and $\ddue$}
Since we will need to apply Borisov-Hua's method, the first thing to do is find out which subsets of the set
of vertices of $\duno$ and $\ddue$ are
forbidden. It turns out that there are eleven of them.
\begin{prop}The forbidden sets for $\mathcal{D}_i$, $i=1,2$ are
$$\emptyset,\;\{3,6\},\;\{4,6\},\;\{3,5\},\;\{1,2,5\},\;\{1,2,4\},\;\{1,2,4,5\},$$
$$\{1,2,3,5\},\;\{1,2,4,6\},\;\{3,5,6\},\;\{3,4,6\}.$$
\end{prop}
\begin{proof}
We know from \cite[Proposition 5.7]{smallpicard} that the aforementioned sets are indeed the forbidden sets
for $\mathcal{D}_2$. Now, being a
forbidden sets depends just upon the primitive collections of a variety and  not by their relations. Thus the
same result is true also for
$\mathcal{D}_1$, since $\mathcal{D}_1$ and $\mathcal{D}_2$ have the same primitive collections.
\end{proof}
\begin{thm}\label{prop:duno}
 The following is a full strongly exceptional sequence for $\duno$:
\begin{gather}\label{eq:seq
d1}\fas{O},\;\fas{O}(Z_4+Z_5),\;\fas{O}(2Z_4+2Z_5),\;\fas{O}(Z_6),\;\fas{O}(Z_5+Z_6),\\
\notag\;\fas{O}(Z_4+Z_5+Z_6),\; 
\fas{O}(Z_4+2Z_5+Z_6),\;\fas{O}(2Z_4+2Z_5+Z_6).\end{gather}
\end{thm}
\begin{proof}
Proposition \ref{full d1} tells us that the sequence \eqref{eq:seq d1} is full. So, in order to prove it is
also strongly exceptional, we have just to
show that the line bundles in the sequence satisfy the required vanishing. As first step we will demonstrate
that the following line bundles are
acyclic:
 \begin{gather} \label{eq:acyclic d1}
\fas{O}(\pm(Z_4+Z_5)),\;\fas{O}(\pm(2Z_4+2Z_5),\;\fas{O}(\pm(Z_5+Z_6)) \fas{O}(\pm(Z_4+Z_5+Z_6)),\;\\ \notag
\;\fas{O}(\pm Z_6),
\fas{O}(\pm(Z_4+2Z_5+Z_6)),\; \fas{O}(\pm(2Z_4+2Z_5+Z_6)),\;\fas{O}(\pm(-Z4-Z_5+Z_6),\;\\ \notag
\fas{O}(\pm(-Z_4+Z_6)),\;\fas{O}(\pm(-2Z_4-2Z_5+Z_6)),\;\fas{O}(\pm(-2Z_4-Z_5+Z_6)),\\ \notag
\fas{O}(\pm Z_5),\;\fas{O}(\pm(Z_4+2Z_5),\;\fas{O}
(\pm Z_4),\;\fas{O}(\pm(2Z_4+Z_5)).
\end{gather}
In order to shorten our argument, we shall split these invertible sheaves into four groups.
\begin{itemize}
 \item[a)] $\fas{O}(Z_4+Z_5),$ $\fas{O}(2Z_4+2Z_5),$ $\fas{O}(Z_6),$ $\fas{O}(Z_5+Z_6),$
$\fas{O}(Z_4+Z_5+Z_6),$\\ $\fas{O}(Z_4+2Z_5+Z_6),$
$\fas{O}(2Z_4+2Z_5+Z_6)$, $\fas{O}(-Z_4+Z_6),$ $\fas{O}(Z_5+Z_6),$\\  $\fas{O}(-2Z_4-Z_5+Z_6),$
$\fas{O}(Z_5),$
$\fas{O}(Z_4+2Z_5),$ $\fas{O}(Z_4),$ $\fas{O}(2Z_4+Z_5);$
\item[b)] $\fas{O}(-Z_4-Z_5),$ $\fas{O}(-2Z_4-2Z_5),$ $\fas{O}(-Z_6),$ $\fas{O}(-Z_4-Z_5-Z_6),$\\
$\fas{O}(-2Z_4-2Z_5-Z_6),$ $\fas{O}(-Z_4-Z_5+Z_6),$
$\fas{O}(2Z_4+2Z_5-Z_6),$\\ $\fas{O}(-2Z_4-2Z_5+Z_6),\;\fas{O}(Z_4+Z_5-Z_6),$
\item[c)] $\fas{O}(-Z_5-Z_6),$  $\fas{O}(-Z_4-2Z_5-Z_6),$ $\fas{O}(Z_4-Z_6),$ $\fas{O}(2Z_4+Z_5-Z_6),$
$\fas{O}(-Z_5),$ $\fas{O}(-Z_4-2Z_5)$, $\fas{O}(-Z_4-2Z_5+Z_6);$    
\item[d)] $\fas{O}(-Z_4),$ $\fas{O}(-2Z_4-Z_5).$  
\end{itemize}
We claim that the line bundles in group a) are line bundle associated to toric effective divisors, whose
coefficients of the principal toric divisors are either zero or one and hence they are acyclic due
Remark \ref{mustata}. Indeed this is straightforward for those line
bundle of the form $\fas{O}(aZ_4+bZ_5+cZ_6)$ with $a,b,c$ already in $\{0,1\}$. Let us see, as an example,
that $-2Z_4-Z_5+Z_6$ is a divisor linearly equivalent to
$Z_3$; all the other are proved by similar reasoning. First of all we want to write the generic divisor
$D=\sum_{\rho=1}^6 a_\rho
Z_\rho$ of $\duno$ in the basis of $\mathrm{Pic}(\duno)$ given by $Z_4$, $Z_5$ and $Z_6$, using the following
relations:
$$Z_1=Z_2=Z_4+Z_5,\quad\quad Z_3=-2Z_4-Z_5+Z_6.$$
We get:
\begin{equation}\label{eq:divisor d1}
 D=(a_1+a_2-2a_3+a_4)Z_4+(a_1+a_2-a_3+a_5)Z_5+(a_3+a_6)Z_6.
\end{equation}
Using the previous formula, we can easily be proved that $-2Z_4+-Z_5+Z_6$ is linearly equivalent to $Z_3$, and
then $\fas{O}(-2Z_4+-Z_5+Z_6)\simeq\fas{O}(Z_3)$. \\
 To prove the
acyclicity of the remaining line
bundles we need to use Borisov-Hua's result.\\
Now denote with $z_i$ for $i=4,5,6$ the coefficient of $Z_i$ in the representation of $D$; all the divisors we
have to check have the coefficient
$z_6\geq -1$, hence these divisors cannot be written in the forbidden forms relative to a set $I$ which does
not have neither 3 nor 6 among its
elements. In fact the coefficient $z_6$ of the divisors in the forbidden forms relative to those sets is
always less or equal to -2. The sets we have
still to check are:
$$\{3,6\},\;\{4,6\},\;\{3,5\},\;\{1,2,3,5\},\;\{1,2,4,6\},\;\{3,5,6\},\;\{3,4,6\}.$$
For similar reasons we know that the line bundles in \eqref{eq:acyclic d1} cannot be put in the forbidden
forms relative to a set $I$ when it:
\begin{enumerate}
 \item contains 3 but neither 1 nor 2 nor 4 (otherwise we will have $z_4\leq -3$);
\item contains 3 but neither 1 nor 2 nor 5 (otherwise we will have $z_5\leq -3$).
\end{enumerate}
After this second cancellation we remain with
$$\{4,6\},\;\{1,2,3,5\},\;\{1,2,4,6\}.$$
From \eqref{eq:divisor d1} we can deduce that
$$a_4-a_3-a_5=z_4-z_5,$$
thus we have
\begin{equation}\label{eq:cases z4-z5}a_4-a_3-a_5=\begin{cases}
               0 & \text{If we have $\fas{O}(D)$ in group b)},\\
               1 & \text{If we have $\fas{O}(D)$ in group c)},\\
               -1 & \text{If we have $\fas{O}(D)$ in group d)}.
              \end{cases}\end{equation}
It follows immediately that neither one of the divisors in those groups can be linearly equivalent to a line
bundle in a forbidden form relative to
$\{4,6\}$ or $\{1,2,4,6\}$, otherwise
it should have $z_4-z_5\geq 2$. It remains to check that the divisors cannot be put in the forbidden forms
relative to $\{1,2,3,5\}$. The forbidden
forms relative to this set have the  difference $z_4-z_5\leq-1$. It follows from this and \eqref{eq:cases
z4-z5} that if $\fas{O}(D)$ can be put in
one of this forbidden forms, then it is  in group d) and $a_3=a_5=0$. But $6\notin\{1,2,3,5\}$, hence we
should have $z_6=a_3+a_6\leq -1$ while none
of the divisor in group d) have a negative $z_6$. Thus we can conclude that all the divisors in groups a), b),
c) and d) are acyclic.\\
To finish the proof we have still to check that the line bundles associated to the following divisors have no
sections.
\begin{itemize}\item[a)] $-Z_6$, $-Z_5-Z_6$, $-Z_4-Z_5-Z_6$, $-Z_4-2Z_5-Z_6$, $-2Z_4-2Z_5-Z_6$,\linebreak
$Z_4+Z_5-Z_6$, $Z_4-Z_6$, $2Z_4+2Z_5-Z_6$,
$2Z_4+Z_5-Z_6$;
\item[b)] $-Z_4-Z_5$, $-2Z_4-2Z_5$, $-Z_5$, $-Z_4-2Z_5$, $-Z_4$, $-2Z_4-Z_5$, $-Z_5$.\end{itemize}
Looking at \eqref{eq:divisor d1} is straightforward to see that all the divisors in group a), which have a
negative $z_6$, are not linearly equivalent
to a toric effective divisor, hence their associated line bundles have not zero-cohomology. The divisors in b)
are the opposite of some positive
divisors and hence they cannot  have sections and the statement isproved.  
\end{proof}



\begin{thm}\label{prop:ddue}
 The following is a full strongly exceptional sequence for $\ddue$:
\begin{gather}\label{eq:seq
d2}\fas{O},\;\fas{O}(Z_4+Z_5),\;\fas{O}(2Z_4+2Z_5),\;\fas{O}(Z_6),\;\fas{O}(Z_5+Z_6),\\
\notag\;\fas{O}(Z_4+Z_5+Z_6),\; 
\fas{O}(Z_4+2Z_5+Z_6),\;\fas{O}(2Z_4+2Z_5+Z_6).\end{gather}
\end{thm}
\begin{proof}
Proposition \ref{full d2} tells us that the sequence \eqref{eq:seq d2} is full. So, in order to prove it also
is strongly exceptional, we have just to
show that the line bundles in the sequence satisfy the required vanishing. As first step we will demonstrate
that the following invertible sheaves,
obtained as a
difference of two line bundles in the sequence, are acyclic.
 \begin{gather} \label{eq:acyclic d2}
\fas{O}(\pm(Z_4+Z_5)),\;\fas{O}(\pm(2Z_4+2Z_5),\;\fas{O}(\pm(Z_5+Z_6)) \fas{O}(\pm(Z_4+Z_5+Z_6)),\;\\ \notag
\;\fas{O}(\pm Z_6),
\fas{O}(\pm(Z_4+2Z_5+Z_6)),\; \fas{O}(\pm(2Z_4+2Z_5+Z_6)),\;\fas{O}(\pm(-Z4-Z_5+Z_6),\;\\ \notag
\fas{O}(\pm(-Z_4+Z_6)),\;\fas{O}(\pm(-2Z_4-2Z_5+Z_6)),\;\fas{O}(\pm(-2Z_4-Z_5+Z_6)),\\ \notag
\fas{O}(\pm Z_5),\;\fas{O}(\pm(Z_4+2Z_5),\;\fas{O}
(\pm Z_4),\;\fas{O}(\pm(2Z_4+Z_5)).
\end{gather}

 As we already did before, it is useful to split these line bundles in
four groups
\begin{itemize}
 \item[a)] $\fas{O}(Z_4+Z_5),$ $\fas{O}(2Z_4+2Z_5),$ $\fas{O}(Z_6),$ $\fas{O}(Z_5+Z_6),$
$\fas{O}(Z_4+Z_5+Z_6),$\linebreak $\fas{O}(Z_4+2Z_5+Z_6),$
$\fas{O}(2Z_4+2Z_5+Z_6)$, $\fas{O}(-Z_4+Z_6),$ $\fas{O}(Z_5+Z_6),$ $\fas{O}(Z_5),$ $\fas{O}(Z_4+2Z_5),$
$\fas{O}(Z_4),$ $\fas{O}(2Z_4+Z_5);$
\item[b)] $\fas{O}(-Z_4-Z_5),$ $\fas{O}(-2Z_4-2Z_5),$ $\fas{O}(-Z_6),$ $\fas{O}(-Z_4-Z_5-Z_6),$
\linebreak$\fas{O}(-2Z_4-2Z_5-Z_6),$
$\fas{O}(-Z_4-Z_5+Z_6),$
$\fas{O}(2Z_4+2Z_5-Z_6),$\linebreak $\fas{O}(-2Z_4-2Z_5+Z_6),\;\fas{O}(Z_4+Z_5-Z_6);$
\item[c)] $\fas{O}(-Z_5-Z_6),$  $\fas{O}(-Z_4-2Z_5+Z_6),$ $\fas{O}(-Z_4-2Z_5-Z_6),$
$\fas{O}(Z_4-Z_6),$\linebreak $\fas{O}(2Z_4+Z_5-Z_6),$
$\fas{O}(-Z_5),$
$\fas{O}(-Z_4-2Z_5);$    
\item[d)] $\fas{O}(-Z_4),$ $\fas{O}(-2Z_4+-Z_5+Z_6),$ $\fas{O}(-2Z_4-Z_5).$  
\end{itemize}
As a first step we  write the generic divisor $D=\sum_{\rho=1}^6 a_\rho Z_\rho$ of $\ddue$ in the basis of
$\mathrm{Pic}(\ddue)$ given by $Z_4$, $Z_5$
and $Z_6$, using the following relations:
$$Z_1=Z_2=Z_4+Z_5,\quad\quad Z_3=-Z_4+Z_6.$$
We get:
\begin{equation}\label{eq:divisor d2}
 D=(a_1+a_2-a_3+a_4)Z_4+(a_1+a_2+a_5)Z_5+(a_3+a_6)Z_6.
\end{equation}
It is quite strightforward to see that all the line bundles in group a) are of the
form $\fas{O}(E)$ with $E$ a linear combination with coefficients 0 or 1 of the principal toric divisors, thus
they are acyclic for Remark \ref{mustata}. It is all the same easy to verify that if the remainig line
bundles are forbidden with respect to a set $I$, then, necessarily, we have $I=\{1,2,3,5\}$ or
$I=\{1,2,4,6\}.$
Now observe that
$$a_4-a_5-a_3=\begin{cases}
               -1& \text{If we have $\fas{O}(D)$ is in group d),}\\
0& \text{If we have $\fas{O}(D)$ is in group b),}\\
1& \text{If we have $\fas{O}(D)$ is in group c).}\\

              \end{cases}
$$
Thus we can also eliminate $\{1,2,4,6\}$, since all the forbidden forms relative to this set have
$a_4-a_3-a_5\geq 2$.\\
Now observe that for the line bundles $\fas{O}(\sum_\rho a_\rho Z_\rho)$ in the forbidden form relative to
$\{1,2,3,5\}$ we have that
$a_4-a_3-a_5\leq-1$. Thus there is just one possibility for them to apply to the line bundles arising from the
full sequence \ref{eq:seq d2}: we
should have $a_4=-1$ while $a_3=a_5=0$; then $z_4-z_5=-1$ and the line bundles must be in group d). But since
$6\notin I$ these line bundle should
have a negative $z_6$ too, and this do not apply to any of the invertible sheaves in group d).
To finish the proof we have still to check that the line bundles associated to the following divisors have no
sections.
\begin{itemize}\item[a)] $-Z_6$, $-Z_5-Z_6$, $-Z_4-Z_5-Z_6$, $-Z_4-2Z_5-Z_6$, $-2Z_4-2Z_5-Z_6$,\linebreak
$Z_4+Z_5-Z_6$, $Z_4-Z_6$, $2Z_4+2Z_5-Z_6$,
$2Z_4+Z_5-Z_6$;
\item[b)] $-Z_4-Z_5$, $-2Z_4-2Z_5$, $-Z_5$, $-Z_4-2Z_5$, $-2Z_4-Z_5$, $-Z_4$, $-Z_5$.
\end{itemize}
Looking at \eqref{eq:divisor d2} is straightforward to see that all the divisors in group a), which have a
negative $z_6$,
are not linearly equivalent to a toric effective divisor, hence their associated line bundles have not
zero-cohomology.The divisors in group b), on
the other hand, are opposite of positive divisors and hence they have no sections. Thus the statement is
proved.   
\end{proof}

\subsection{$\mathcal{E}_1$, $\edue$ and $\equat$}
Again we need to find the forbidden sets. It can be proved the following statement:
\begin{prop}The forbidden sets for $\mathcal{E}_i$, $i=1,2,4$ are
$$\emptyset,\;\{2,4\},\;\{3,5\},\;\{1,3\},\;\{2,5\},\;\{1,4\},\;\{6,7\},$$
$$\{2,4,5\},\;\{2,4,1\},\;\{3,5,1\},\;\{3,5,2\},\;\{1,3,4\},$$
$$\{1,3,6,7\},\;\{3,5,6,7\},\;\{2,4,6,7\},\;\{1,4,6,7\},\;\{2,5,6,7\},$$
$$\{1,3,5,6,7\},\;\{2,4,5,6,7\},\;\{1,3,5,6,7\},\;\{1,2,4,6,7\},$$
$$\{1,3,4,6,7\},\;\{2,3,5,6,7\},\;\{1,2,3,4,5\}$$
\end{prop}
\begin{proof}
Certainly the simplicial complex associated to a set with just one element has a trivial reduced homology.
Since the faces have all trivial homology,
the only two-elements forbidden sets are the primitive  collection. We want to show that the forbidden sets of
cardinality three are precisely the
unions of two primitive collections and so they are
$$\{2,4,5\},\;\{2,4,1\},\;\{3,5,1\},\;\{3,5,2\},\;\{1,3,4\}.$$
Indeed in a three element forbidden set can either contain:
\begin{enumerate}
 \item zero primitive collections;
\item one primitive collection;
\item two primitive collections.
\end{enumerate}
If the first is the case, than the set is a maximal cone and hence it has a trivial homology.
If we are in the hypothesis of the second case, then the chain complex associated to the simplicial complex
will look like:
$$0\longrightarrow K^2\longrightarrow K^3\longrightarrow K\longrightarrow 0$$
that is an exact sequence. Thus the simplicial complex has a trivial reduced homology. Finally, if the set $I$
is the union of two primitive
collections, then its associated chain complex will be
$$0\longrightarrow K\longrightarrow K^3\longrightarrow K\longrightarrow 0$$
and  hence it is a forbidden set.\\
\indent\textit{Claim 1}: The four-elements forbidden sets are the complementary of the three elements
forbidden sets. Indeed let us suppose that $I$
is
the complementary of a maximal cone. Since all maximal cone are generated by either $v_6$ o $v_7$, a
complementary of a cone is necessarily of the
following form
$$\{a_1,a_2,b,c\}$$
with $\{a_1,a_2\}$ the only one primitive collection contained in $I$. Thus its associated chain complex is
$$0\longrightarrow K^2\longrightarrow K^5\longrightarrow K^4\longrightarrow K\longrightarrow 0$$
that is obviously exact.\\
Suppose now that $I$ is the complementary of $J$, a set of the form $\{p_1,p_2,q\}$ that contains just one
primitive collection, namely $\{p_1,p_2\}$.
Since all such $I$ count either 6 or 7 among their element, we can consider two different cases:
\begin{enumerate}\item both 6 and 7 are in $I$,
\item $\{6,7\}\not\subseteq I$
\end{enumerate}
 If both $6$ and $7$ are in $J$, then $I$ will contain three primitive collections and its associated chain
complex will be
$$0\longrightarrow K^3\longrightarrow K^4\longrightarrow K\longrightarrow 0$$
that is exact, and hence $I$ is not forbidden. It can happen that just one among 6 and 7 is an element of $J$.
Under this assumption $I$ will contain
three primitive collections and will exist one of its elements that will not belong to any of these. In this
case the chain complex associate to $I$ will
be
$$0\longrightarrow K\longrightarrow K^4\longrightarrow K^4\longrightarrow K\longrightarrow 0$$
that is exact, as required.
To prove the claim it remain to show that the complementary of the forbidden sets of cardinality equal to
three are still forbidden sets. But it can be
easily seen that these sets are of the form $\{a_1,a_2,b_1,b_2\}$ with $\{a_1,a_2\}$ and $\{b_1,b_2\}$ the
only primitive collections in $I$. Hence the
associated chain complex will look like
$$0\longrightarrow K^4\longrightarrow K^4\longrightarrow K\longrightarrow 0$$
and it is obviously not exact.\\
\indent\textit{Claim 2} The forbidden sets of cardinality five are the complementaries of the primitive
collections.\\
Let us suppose that the set $I$ is the complementary of a face, then there are two cases: or $\{6,7\}\subseteq
I$, or just one among 6 and 7 is in
$I$. If we are in the first situation, then it can be seen that $I$ contains two disjoint primitive
collections and its associated chain complex is
$$0\longrightarrow K^4\longrightarrow K^8\longrightarrow K^5\longrightarrow K\longrightarrow 0$$
that is exact. Now assume that not both 6 or 7 are in $I$. Then $I$ will contain three primitive collection
that will cover just four of its five
elements. Its associated chain complex will be
$$0\longrightarrow K^3\longrightarrow K^7\longrightarrow K^5\longrightarrow K\longrightarrow 0.$$
Now suppose, viceversa, that  $I$ is the complementary of a primitive collection $J$. If $J=\{6,7\}$ then the
chain complex associated to  $I$ is
$$0\longrightarrow K^5\longrightarrow K^5\longrightarrow K\longrightarrow 0$$
that is not exact.\\
If, otherwise $J$ is different from $\{6,7\}$, then $I$ contains exactly 3 primitive collections: two of them
cover three elements of the set, while
the last one
is disjoint from the previous. Knowing this data is straightforward to see that the chain complex  associated
to $I$ is
$$0\longrightarrow K^2\longrightarrow K^7\longrightarrow K^5\longrightarrow K\longrightarrow 0.$$
\\

It remains to prove that do not exist forbidden sets of cardinality $6$. Again we can split the proof in two
cases, depending on whether $I$ contains
$\{6,7\}$ or not. In the first case $I$  will contain four primitive collections: $\{6,7\}$ and other three
that will cover the remaining four
elements
of $I$. Knowing this it is easy to check that the chain complex associated to $I$ is
$$0\longrightarrow K^6\longrightarrow K^{11}\longrightarrow K^6\longrightarrow K\longrightarrow0.$$
If otherwise $I$ does not contain $\{6,7\}$, thaen it will contain five primitive relations that will cover
five of its 6 elements. Its associated
chain
complex will  be of the form
$$0\longrightarrow K^5\longrightarrow K^{10}\longrightarrow K^6\longrightarrow K\longrightarrow 0$$
that is again exact, and hence the proof is complete.

\end{proof}

Using the previous result we are now able to prove the next three propositions.
\begin{thm}\label{prop:euno}
 The following is a full strongly exceptional sequence of line bundles for $\euno$.
\begin{gather}\label{eq:sequence
e1}\mathcal{O},\;\mathcal{O}(Z_7),\;\mathcal{O}(Z_4),\;\mathcal{O}(Z_4+Z_7),\;\mathcal{O}(Z_4+Z_5),\;\mathcal{
O}
(Z_1+Z_5+2Z_7),\\ \notag
\mathcal{O}(Z_4+Z_5+Z_7),\;\mathcal{O}(Z_1+Z_4+Z_5+Z_7),\;\mathcal{O}(Z_1+Z_4+Z_5+2Z_7).
\end{gather}
\end{thm}
\begin{proof}
 We already know, by Bondal's method, that this sequence of line bundles generates the bounded derived
category of $\euno$. \\
Our first step will be to check if the following line bundles satisfy the required acyclicity.
\begin{gather}\label{eq:acyclic e1}
\odiv{Z_7},\;\odiv{Z_4},\;\odiv{Z_4+Z_7},\;\odiv{Z_4+Z_5}\odiv{Z_4+Z_5},\\
\notag\;\odiv{Z_1+Z_5+2Z_7},\;\odiv{Z_4+Z_5+Z_7},\;\odiv{Z_1+Z_4+Z_5+Z_7},\\
\notag\;\odiv{Z_1+Z_4+Z_5+2Z_7},\;\odiv{Z_4-Z_7},\;\odiv{Z_4+Z_5-Z_7},\\\notag\;\odiv{Z_1-Z_5+Z_7},\;\odiv{
Z_1+Z_4+Z_5},\;\odiv{Z_5},\;\odiv{
Z_1-Z_4+Z_5+2Z_7},\\ \notag \odiv{Z_5+Z_7},\; \odiv{Z_5-Z_7},\;\odiv{Z_1-Z_4+Z_5+Z_7},\;\odiv{Z_5-Z_7},\\
\notag
\odiv{Z_1-Z_4+2Z-7},\;\odiv{Z_1+Z_7},\;\odiv{Z_1+2Z_7},\;\odiv{-Z_1+Z_4-Z_7},\;\\ \notag\odiv{Z_1}.
\end{gather}
As before, in order to easy the computation, we divide these line bundles in groups:
\begin{enumerate}\item[a)]
$\mathcal{O}(Z_7),$\;$\mathcal{O}(Z_4),$\;$\mathcal{O}(Z_4+Z_7),$\;$\mathcal{O}(Z_4+Z_5),$\;$\mathcal{O}
(Z_4+Z_5+Z_7)$,\linebreak
$\mathcal{O}(Z_1+Z_5+2Z_7),$\;$\mathcal{O}(Z_1+Z_4+Z_5+Z_7),$\;$\mathcal{O}(Z_1+Z_4+Z_5+2Z_7)$,\linebreak
$\mathcal{O}(Z_4-Z_7)$, $\mathcal{O}(Z_4+Z_5-Z_7)$, $\mathcal{O}(Z_1+Z_5)$, $\mathcal{O}(Z_1+Z_4+Z_5)$,
$\mathcal{O}(Z_5)$,\linebreak
$\mathcal{O}(Z_5+Z_7)$,
$\mathcal{O}(Z_1+Z_5+2Z_7)$, $\mathcal{O}(-Z_1+Z_4-Z_7)$, $\mathcal{O}(Z_1+2Z_7)$,\linebreak
$\mathcal{O}(-Z_1+Z_4)$, $\mathcal{O}(-Z_1+Z_4-Z_7)$,
$\mathcal{O}(Z_1),\;\fas{O}(Z_1+Z_5+Z_7),\;\fas{O}(Z_1+Z_7)$;
\item[b)] $\mathcal{O}(-Z_4-Z_5)$, $\mathcal{O}(-Z_4-Z_5-Z_7)$, $\mathcal{O}(-Z_1-Z_5-Z_7)$,
$\mathcal{O}(-Z_1-Z_5-2Z_7)$,
$\mathcal{O}(-Z_1-Z_4-Z_5-Z_7)$, $\mathcal{O}(-Z_1-Z_4-Z_5-2Z_7)$, $\mathcal{O}(-Z_4-Z_5+Z_7)$,
$\mathcal{O}(-Z_1-Z_5)$, $\mathcal{O}(-Z_1-Z_4-Z_5)$,
$\mathcal{O}(-Z_5)$,  $\mathcal{O}(-Z_1+Z_4-Z_5-2Z_7)$,\linebreak $\mathcal{O}(-Z_5-Z_7)$,
$\mathcal{O}(-Z_1+Z_4-Z_5-Z_7)$, $\mathcal{O}(-Z_5+Z_7)$;
\item[c)] $\mathcal{O}(-Z_4+Z_7)$, $\mathcal{O}(-Z_4)$, $\mathcal{O}(Z_1-Z_4+Z_7)$,
$\mathcal{O}(-Z_1-Z_7)$,\linebreak $\mathcal{O}(Z_1-Z_4)$,
$\mathcal{O}(-Z_4-Z_7)$, $\mathcal{O}(-Z_1)$, $\fas{O}(Z_1-Z_4+2Z_7)$,  $\fas{O}(-Z_7)$;  
\item[d)] $\mathcal{O}(Z_1-Z_4+Z_5+Z_7)$,  $\mathcal{O}(Z_1-Z_4+Z_5+2Z_7)$, 
$\mathcal{O}(Z_1-Z_4+Z_5)$,\linebreak $\fas{O}(-Z_1-2Z_7)$;
\item[e)] $\mathcal{O}(-Z_1+Z_4-2Z_7)$; 
\item[f)] $\mathcal{O}(Z_5-Z_7)$.  
\end{enumerate}
Given a divisor $D=z_1Z_1+z_4Z_4+z_5Z_5+z_7Z_7$ on $\euno$, written in the basis\linebreak $(Z_1,Z_4,Z_5,Z_7)$
of $\mathrm{Pic}(\euno)$ we want to
find
conditions for another divisor\linebreak $D'=\sum_{\rho=1}^7a_\rho Z_\rho$ to be linear equivalent to $D$. In
order to do that we write $D'$ in the
given basis
using the following relations among the principal toric divisors:
$$ Z_2=-Z_1+Z_4-Z_7,\quad Z_3=Z_1+Z_5+Z_7,$$
$$Z_6=Z_7.$$
We gain
\begin{equation}\label{eq:divisor e1}
D'\simeq_{\mathrm{lin}} (a_1-a_2+a_3)Z_1 +(a_2+a_4)Z_4+ (a_5+a_3)Z_5+(-a_2+a_3+a_6+a_7)Z_7.
\end{equation}
Now, imposing equality among the coefficient of $D'$ and $D$ we have the conditions we were seeking.\\
Using \eqref{eq:divisor e1}, it can easily be seen that all the line bundles in group a) are line bundles
associated to  toric effective divisors with coefficient $a_i\in\{1,0\}$, and
hence are acyclic. In order to check the acyclicity of the other line bundles we are going to use
Borisov-Hua's criterion. Observe that:
\begin{enumerate}
 \item The forbidden forms relative to a set $I$ which does not contain neither 2 nor 4 as elements have the
coefficient $z_4\leq -2$.
\item The forbidden forms relative to a set $I$ which does not contain neither 3 nor 5 satisfy $z_5\leq -2$.
\item The forbidden forms relative to a set $I$ such that $2\in I$, but neither 1, nor 3 are in $I$ satisfy
$z_1\leq-2$.
\item The forbidden forms relative to a set $I$ with $2\in I$ and such that its complementary $I'$ contains
$\{3,6,7\}$ as a subset have $z_7\leq
-3$.
\end{enumerate}
Thus we can eliminate all the forbidden sets satisfying condition 1)-4). The set we have still to check are:
$$\{2,3,5\},\;\{1,3,4\}\;\{1,3,4,6,7\},\;\{2,3,5,6,7\},\;\{1,2,3,4,5\}.$$
We can eliminate $\{1,3,4\}$ and $\{1,3,4,6,7\}$ in the following way: suppose the some of the line bundles we
are working with can be put in one of
the forbidden forms relative to these sets. Since both  sets $\{1,3,4\}$ and $\{1,3,4,6,7\}$ contains 1 and 3
but not 2 it follows that
$z_1=a_1-a_2+a_3\geq 1$. The only possibility for our divisors is $z_1=1$ and hence $a_1=a_3=0$ and $a_2=-1$.
But $5\notin\{1,3,4\}\cup\{1,3,4,6,7\}$,
thus $a_3=0$ would imply that $z_5<0$. Hence these line bundles should be in group b). But all the divisors in
group b) have a non positive $z_1$.\\
Now we will show that none among the divisors in b)-f) can be put in a forbidden form relative to
$I=\{2,3,5\}$. Surely, since both $a_3$ and $a_5$
are
positive, none of the divisors in b) (which have a negative $z_5$) is in the forbidden form relative to $I$.
If we have a line bundle associated to a
divisor $D$ whose coefficient $z_5$ is null, then it can be put in one of the forbidden form relative to $I$
if $a_3=a_5=0$. As a consequence
$z_7=-a_2+a_3+a_6+a_7\leq-2$. Thus the line bundles in c) (which have $z_5=0$ and $z_7\geq-1$) cannot be put
in the required forbidden form. Let us
see that neither one of the line bundles in e) can be of the forbidden form relative to $I$: if this would be
the case, then we will have that $a_2=0$
and,
since $4\notin I$, $z_4\leq-1$ that is impossible. Now we want to show that the line bundles in d) and f) are
not in the forbidden form relative to
$I$. In this case $a_3$ can be both 0 or -1. In any case we will have that $z_7\leq -1$, and this is
sufficient to eliminate all the bundles in d). As
before we eliminate the invertible sheaf in e) because its $z_4$ is not negative.\\
In order to eliminate the set $I=\{1,2,3,4,5\}$ we proceed in a similar way. Observe now that again $a_3$ can
be $0$ or $1$.
If $a_3$ is 0, then $z_7\leq-2$ and $z_1\geq 0$ that is impossible, because the only invertible sheaf with
$z_7=-2$ is the one in e) and have
$z_1=-1$. Then $a_3=1$. But in this case we have $z_5\geq1$, $z_7\leq-1$ and $z_1\geq 1$ that is again
impossible.
It remains to check that none of the divisors in the list can be put in the forbidden form relative to
$I=\{2,3,5,6,7\}$. Again $a_3=0,1$. If $a_3=0$,
then $z_1\leq -1$ it follows that $z_1=-1 $ and $a_2=0$. As a consequence $z_4\leq-1$ but none of the divisor
we have has both a negative $z_1$ and a
negative $z_4$. Now suppose that $a_3=1$. Then $z_5=1$ and $z_1\leq 0$. All the line bundles with a positive
$z_5$ (group d) and f)) have a
non-negative $z_1$. It follows that $z_1=0$ and, in particular $z_4\leq -1$ but this is impossible.\\
Finally, to finish proving the proposition, we have to check that the following divisors are not linearly
equivalent to any toric effective divisor.
\begin{enumerate}
 \item[a)] $-Z_4$, $-Z_4-Z_7$, $-Z_4-Z_5$, $-Z_4-Z_5-Z_7$, $-Z_1-Z_5-2Z_7$,\linebreak $-Z_1-Z_4-Z_5-Z_7$,
$-Z_1-Z_4-Z_5-2Z_7$, $-Z_4+Z_7$,
$-Z_4-Z_5+Z_7$,
$-Z_1-Z_5+Z_7$, $-Z_1-Z_4-Z_5$, $-Z_5$, $-Z_5-Z_7$, $-Z_1+Z_4-Z_5-2Z_7$, \linebreak
$-Z_1-Z_5-Z_7$, $-Z_5+Z_7$, $-Z_1+Z_4-Z_5-Z_7$, $-Z_1-Z_5$, $Z_1-Z_4+Z_7$;
\item[b)] $-Z_7$, $-Z_1+Z_4-2Z_7$, $-Z_1-Z_7$, $-Z_1-2Z_7$, $-Z_1$.
\end{enumerate}
Observe that the divisors in a), which have a negative $z_4$ or a negative $z_5$ cannot possibly be linear
equivalent to a toric effective divisor.
For what it concerns group b), it can be deduced from \eqref{eq:divisor e1} that it is a necessary condition
in order for a divisor $D$ to be linearly
equivalent to a toric effective divisor that $z_1\geq-a_2\geq-z_4$ and $z_7\geq-a_2\geq-z_4$. It is easy to
see that none of the divisors in group b)
satisfies this condition.
\end{proof}

\begin{thm}\label{prop:edue}
 The following is a full strongly exceptional sequence for $\edue$:
$$\mathcal{O},\;\mathcal{O}(Z_7),\;\mathcal{O}(Z_4),\;\mathcal{O}(Z_1+Z_5),\;\mathcal{O}(Z_1+Z_5+Z_7),
\;\mathcal{O}(Z_4+Z_7),\;\mathcal{O}(Z_4+Z_5),
$$
$$
\mathcal{O}(Z_4+Z_5+Z_7),\;\mathcal{O}(Z_1+Z_4+Z_5),\;\mathcal{O}(Z_1+Z_4+Z_5+Z_7).
$$
\end{thm}
\begin{proof}
 We already know that the aforementioned sequence is full. Thus we have just to show that it is strongly
exceptional. First of all we  will  prove the
vanishing of the higher  cohomology of the following line bundles:
\begin{gather}
 \odiv{Z_7},\;\odiv{Z_4},\;\odiv{Z_1+Z_5},\;\odiv{Z_1+Z_5+Z_7},\;\odiv{Z_4+Z_7},\;\\
\notag\odiv{Z_4+Z_5},\;\odiv{Z_4+Z_5+Z_7},\;\odiv{Z_1+Z_4+Z_5},\;\\\notag\odiv{Z_1+Z_4+Z_5+Z_7},\;\odiv{
Z_4-Z_7},\;\odiv{Z_1+Z_5-Z_7},\;\\\notag\odiv{
Z_4+Z_5-
Z_7},\;\odiv{Z_1+Z_4+Z_5-Z_7},\;\odiv{Z_1+Z_5-Z_4},\\\notag\odiv{Z_5},\;\odiv{Z_5+Z_7},\;\odiv{Z_1-Z_4+Z_5-Z_7
},\;\odiv{Z_4-Z_1},\\\notag\odiv{
Z_1-Z_4-Z_7},\;
\odiv{Z_4+Z_7},\;\odiv{Z_1-Z_4+Z_7},\;\odiv{Z_5-Z_7},\\
\notag\odiv{Z_1+Z_7},\; \odiv{Z_1-Z_7},\;\odiv{Z_1},\;\mathcal{O}(\pm(-Z_1+Z_4-Z_5+Z_7)).
\end{gather}
As usual, it is better to split all these invertible sheaves into four groups.
\begin{itemize}
 \item[a)]
$\mathcal{O},\;\mathcal{O}(Z_7),\;\mathcal{O}(Z_4),\;\mathcal{O}(Z_1+Z_5),\;\mathcal{O}(Z_1+Z_5+Z_7),
\;\mathcal{O}(Z_4+Z_7),$\linebreak
$\mathcal{O}(Z_4+Z_5),$ $
\mathcal{O}(Z_4+Z_5+Z_7),\;\mathcal{O}(Z_1+Z_4+Z_5),\;\mathcal{O}(Z_1+Z_4+Z_5+Z_7),$ $\mathcal{O}(Z_5)$,
$\mathcal{O}(Z_5+Z_7)$,
$\mathcal{O}(Z_4+Z_7)$, $\mathcal{O}(Z_1)$, $\mathcal{O}(Z_1+Z_7)$, $\mathcal{O}(Z_4-Z_7)$,\linebreak
$\mathcal{O}(-Z_1+Z_4)$,
$\mathcal{O}(-Z_1+Z_4-Z_7)$,
$\mathcal{O}(-Z_1+Z_4+Z_7)$;
\item[b)] $\mathcal{O}(-Z_1+Z_4-Z_5)$, $\mathcal{O}(-Z_1+Z_4-Z_5+Z_7)$,
$\mathcal{O}(-Z_1+Z_4-Z_5-Z_7)$,\linebreak $\mathcal{O}(-Z_5)$,
$\mathcal{O}(-Z_1-Z_5)$,
$\mathcal{O}(-Z_1-Z_5-Z_7)$, $\mathcal{O}(-Z_1-Z_5+Z_7)$,\linebreak $\mathcal{O}(-Z_5-Z_7)$,
$\mathcal{O}(-Z_5+Z_7)$,  $\mathcal{O}(-Z_4-Z_5)$,
$\mathcal{O}(-Z_4-Z_5-Z_7)$,\linebreak $\mathcal{O}(-Z_1-Z_4-Z_5)$, $\mathcal{O}(-Z_1-Z_4-Z_5-Z_7)$,
$\mathcal{O}(-Z_4-Z_5+Z_7)$,\linebreak
$\mathcal{O}(-Z_1-Z_4-Z_5+Z_7)$;
\item[c)] $\mathcal{O}(+Z_1+Z_4+Z_5-Z_7)$, $\mathcal{O}(Z_4+Z_5-Z_7)$, $\mathcal{O}(Z_5-Z_7)$,
$\mathcal{O}(Z_1+Z_5-Z_7)$,
$\mathcal{O}(Z_1-Z_4+Z_5)$, $\mathcal{O}(Z_1-Z_4+Z_5+Z_7)$, $\mathcal{O}(Z_1-Z_4+Z_5-Z_7)$;
\item[d)] $\mathcal{O}(-Z_7)$, $\mathcal{O}(-Z_1-Z_7)$, $\mathcal{O}(-Z_1)$, $\mathcal{O}(Z_1-Z_7)$,
$\mathcal{O}(-Z_1+Z_7)$, $\mathcal{O}(-Z_4)$,
$\mathcal{O}(-Z_4-Z_7)$, $\mathcal{O}(-Z_4+Z_7)$, $\mathcal{O}(Z_1-Z_4)$,
$\mathcal{O}(Z_1-Z_4-Z_7)$,\linebreak $\mathcal{O}(Z_1-Z_4+Z_7)$.
$\mathcal{O}(-Z_1-Z_5+Z_7)$, $\mathcal{O}(-Z_5-Z_7)$,
$\mathcal{O}(Z_1-Z_4-Z_7)$,
$\mathcal{O}(-Z_1-Z_4-Z_5-Z_7)$,

$\mathcal{O}(Z_1-Z_4+Z_5-Z_7)$.
\end{itemize}
Let $D=\sum_{\rho =1}^7 a_\rho Z_\rho$ be any divisor on $\edue$. We want to write it in the basis given by
$Z_1,\;Z_4,\;Z_5$ and $Z_7$ using the
following relations:
$$Z_2=-Z_1+Z_4-Z_7,\quad Z_3=Z_1+Z_5,\quad Z_6=Z_7.$$
We get
\begin{equation}\label{eq:basis
e2}D\simeq_{\mathrm{lin}}(a_1+a_3-a_2)Z_1+(a_3+a_5)Z_5+(a_2+a_4)Z_4+(a_6+a_7-a_2)Z_7.\end{equation}
We indicate with $(z_1,z_4,z_5,z_7)$ the coordinate of $D$ in the chosen basis.\\
We can easily see that all the line bundles in a) can be associated to a toric effective divisor with the
coefficents $a_i\in\{1,0\}$  (and hence acyclic due to Remark \ref{mustata}). For
example$$-Z_1+Z_4+Z_7\simeq_{\mathrm{lin}} Z_2+Z_6+Z_7.$$
For all the other divisors we need Borisov-Hua criterion.\\
It is easy to see that the remaining line bundles cannot be forbidden with respect to a set $I$ unless $I$ is
among the following three sets:
 $$\{1,3,4\},\;\{1,3,4,6,7\},\;\{2,3,5,6,7\}.$$
First of all we eliminate $\{1,3,4\},\;\{1,3,4,6,7\}$: in both these sets appear indeces 1 and 3 and does not
appear 2. This means that coefficient
$z_1$ is greater or equal to 1. The only possibility is that $z_1=1$, hence $a_1=a_3=0$ and $a_2=-1$. Since
neither $5$ is in the previous sets, then
$-1\leq z_5=a_3+a_5\leq -1$. The divisors with $z_5=-1$ are the ones of group b). You can observe that neither
one of these has a positive $z_1$.\\
It remains to eliminate $I=\{2,3,5,6,7\}$.\\
All the divisor of the group b) cannot be of the forbidden form relative to this set. In fact all this
divisors have $z_5=-1$ and the forbidden
divisors relative to $I$ have $z_5\geq 0$. Neither the divisor of group d), that have $z_5=0$ cannot be of the
forbidden forms relative to $I$.
Indeed, if we impose to the forbidden forms the condition to have $z_5=0$ we get that, necessarily,
$a_3=a_5=0$. Thus $a_1=z_1+a_2\geq z_1$. Since
$a_1\leq -1$ we get that $a_1=z_1=-1$. But, on the other side we have that $z_4=a_2+a_4\leq -1$. None of the
divisor of group d) have both $z_1$ and
$z_4$ negative. Finally we can see that also the divisors of group c) are not in of the forbidden forms
relative to $I$: since both $a_3$ and $a_5$
are non negative, we have that $a_3\leq 1$. It follows that $z_1=a_1+a_3-a_2\leq 0$. But all the divisor in c)
have a non negative $z_1$ coefficient,
thus there is just one possibility: $z_1=0$. This yields that $a_1=-1$, $a_3=1$ and  $a_2=0$. The last
equality implies that $z_4=a_4\leq-1$, but this
cannot be since all the  divisors in c)  who have $z_1=0$ have a non negative $z_4$ too.\\
\indent Now, to prove the statement, we just need to show that the following divisors are not linearly
equivalent to a toric effective divisor.
\begin{itemize}
 \item[a)]$-Z_4$, $-Z_4-Z_7$, $-Z_4-Z_5$, $-Z_4-Z_5-Z_7,$ $-Z_1-Z_4-Z_5,$\linebreak $-Z_1-Z_4-Z_5-Z_7,$
$-Z_4+Z_7,$ $-Z_4-Z_5+Z_7,$
$-Z_1-Z_4-Z_5+Z_7,$\linebreak
$Z_1-Z_4+Z_5-Z_7$, $Z_1-Z_4$, $Z_1-Z_4-Z_7$, $-Z_1-Z_4$, $Z_1+Z_5-Z_4$,\linebreak $Z_1+Z_7-Z_4$, $-Z_4+Z_7$;
\item[b)] $-Z_1-Z_5$, $-Z_1-Z_5-Z_7,$ $-Z_1-Z_5+Z_7,$ $-Z_1+Z_4-Z_5,$ $-Z_5,$ $-Z_5-Z_7$, $-Z_1-Z_5-Z_7$,
$-Z_5+Z_7$;
\item[c)] $-Z_7,$ $-Z_1$, $-Z_1-Z_7$, $Z_7-Z_1$.
\end{itemize}
Looking at \eqref{eq:basis e2} it is obvious that the divisors in a) and b), which have or $z_4=-1$ or
$z_5=-1$ cannot be toric effective. The
divisors in the last group have $z_4=0$ and or $z_1$ or $z_7$ equal to -1. If they were equivalent to a toric
effective divisor, then $a_2>0$ but this
is impossible since we would have $z_4=a_2+a_4>0$.
\end{proof}
\begin{thm}\label{prop:equat}
 The following is a full strongly exceptional sequence for $\equat$:
$$\mathcal{O},\;\mathcal{O}(Z_7),\;\mathcal{O}(Z_4),\;\mathcal{O}(Z_1+Z_5),\;\mathcal{O}(Z_1+Z_5+Z_7),
\;\mathcal{O}(Z_4+Z_7),\;\mathcal{O}(Z_4+Z_5),
$$
$$
\mathcal{O}(Z_4+Z_5+Z_7),\;\mathcal{O}(Z_1+Z_4+Z_5),\;\mathcal{O}(Z_1+Z_4+Z_5+Z_7).
$$
\end{thm}
\begin{proof}
 We already know that the aforementioned sequence is full. Thus we have just to show that it is strongly
exceptional. First of all we  will  prove the
acyclicity of the following line bundles:
\begin{gather}
 \odiv{Z_7},\;\odiv{Z_4},\;\odiv{Z_1+Z_5},\;\odiv{Z_1+Z_5+Z_7},\;\odiv{Z_4+Z_7},\;\\
\notag\odiv{Z_4+Z_5},\;\odiv{Z_4+Z_5+Z_7},\;\odiv{Z_1+Z_4+Z_5},\;\\\notag\odiv{Z_1+Z_4+Z_5+Z_7},\;\odiv{
Z_4-Z_7},\;\odiv{Z_1+Z_5-Z_7},\;\\\notag\odiv{
Z_4+Z_5-
Z_7},\;\odiv{Z_1+Z_4+Z_5-Z_7},\;\odiv{Z_1+Z_5-Z_4},\;\\\notag\odiv{Z_5},\;\odiv{Z_5+Z_7},\;\odiv{
Z_1-Z_4+Z_5-Z_7},\;\odiv{Z_4-Z_1},\;
\\\notag\odiv{Z_1-Z_4-Z_7
},\;\odiv{Z_4+Z_7},\;\odiv{Z_1-Z_4+Z_7},\;\odiv{Z_5-Z_7},\;\\
\notag\odiv{Z_1+Z_7},\; \odiv{Z_1-Z_7},\;\odiv{Z_1},\;\mathcal{O}(\pm(-Z_1+Z_4-Z_5+Z_7)).
\end{gather}
As in the previous case we shall split all the line bundles above into four groups.
\begin{itemize}
\item[a)]
$\mathcal{O},\;\mathcal{O}(Z_7),\;\mathcal{O}(Z_4),\;\mathcal{O}(Z_1+Z_5),\;\mathcal{O}(Z_1+Z_5+Z_7),
\;\mathcal{O}(Z_4+Z_7),$\linebreak
$\mathcal{O}(Z_4+Z_5),$ $
\mathcal{O}(Z_4+Z_5+Z_7),\;\mathcal{O}(Z_1+Z_4+Z_5),\;\mathcal{O}(Z_1+Z_4+Z_5+Z_7),$ $\mathcal{O}(Z_5)$,
$\mathcal{O}(Z_5+Z_7)$,
$\mathcal{O}(Z_4+Z_7)$, $\mathcal{O}(Z_1)$, $\mathcal{O}(Z_1+Z_7)$, $\mathcal{O}(-Z_1+Z_4)$,\linebreak 
$\mathcal{O}(-Z_1+Z_4+Z_7)$
$\mathcal{O}(+Z_1+Z_4+Z_5-Z_7)$, $\mathcal{O}(Z_4+Z_5-Z_7)$, $\mathcal{O}(Z_5-Z_7)$,
$\mathcal{O}(Z_1+Z_5-Z_7)$;
\item[b)] $\mathcal{O}(-Z_1+Z_4-Z_5)$, $\mathcal{O}(-Z_1+Z_4-Z_5+Z_7)$,
$\mathcal{O}(-Z_1+Z_4-Z_5-Z_7)$,\linebreak $\mathcal{O}(-Z_5)$,
$\mathcal{O}(-Z_1-Z_5)$,
$\mathcal{O}(-Z_1-Z_5-Z_7)$, $\mathcal{O}(-Z_1-Z_5+Z_7)$,\linebreak $\mathcal{O}(-Z_5-Z_7)$,
$\mathcal{O}(-Z_5+Z_7)$,  $\mathcal{O}(-Z_4-Z_5)$,
$\mathcal{O}(-Z_4-Z_5-Z_7)$,\linebreak $\mathcal{O}(-Z_1-Z_4-Z_5)$, $\mathcal{O}(-Z_1-Z_4-Z_5-Z_7)$,
$\mathcal{O}(-Z_4-Z_5+Z_7)$,\linebreak
$\mathcal{O}(-Z_1-Z_4-Z_5+Z_7)$;
\item[c)]  $\mathcal{O}(Z_1-Z_4+Z_5)$, $\mathcal{O}(Z_1-Z_4+Z_5+Z_7)$, $\mathcal{O}(Z_1-Z_4+Z_5-Z_7)$;
\item[d)] $\mathcal{O}(-Z_7)$, $\mathcal{O}(-Z_1+Z_4-Z_7)$, $\mathcal{O}(Z_4-Z_7)$, $\mathcal{O}(-Z_1-Z_7)$,
$\mathcal{O}(-Z_1)$,\linebreak
$\mathcal{O}(Z_1-Z_7)$, $\mathcal{O}(-Z_1+Z_7)$, $\mathcal{O}(-Z_4)$, $\mathcal{O}(-Z_4-Z_7)$,
$\mathcal{O}(-Z_4+Z_7)$,\linebreak
$\mathcal{O}(Z_1-Z_4)$,
$\mathcal{O}(Z_1-Z_4-Z_7)$, $\mathcal{O}(Z_1-Z_4+Z_7)$.

\end{itemize}
As in the previous cases, our first step will be to write the generic divisor $D=\sum_{\rho =1}^7 a_\rho
Z_\rho$ in the basis  $\mathrm{Pic}(\equat)$
given by $Z_1,\;Z_4,\;Z_5$ and $Z_7$ using the following relations among the generators:
$$Z_2=-Z_1+Z_4,\quad Z_3=Z_1+Z_5-Z_7,\quad Z_6=Z_7.$$
We get
\begin{equation}\label{eq:basis
e4}D\simeq_{\mathrm{lin}}(a_1+a_3-a_2)Z_1+(a_3+a_5)Z_5+(a_2+a_4)Z_4+(a_6+a_7-a_3)Z_7.\end{equation}
As usual we denote with $(z_1,z_4,z_5,z_7)$ the coordinate of $D$ in the chosen basis.\\
A tedious computation shows that all the line bundle in a) are line bundles associated to a divisor whose
coefficents $a_i$ are either 0 or 1 (and hence acyclic thanks to
Remark \ref{mustata}).
For all the other divisors we need Borisov-Hua criterion.\\
As in the case of $\edue$ it can be observed that none of the line bundles above can be of any of the
forbidden forms relative to a forbidden set $I$ unless $I=
=\{1,3,4,6,7\}.$\\
Observe that 1 and 3 are among the elements of $I$ and $2\notin I$. This implies that $z_1=1$, hence
$a_1=a_3=0$ and $a_2=-1$. Since $5\notin I$
$-1\leq z_5=a_3+a_5\leq -1$. The divisors with $z_5=-1$ are the ones of group b) and none of these has a
positive $z_1$.\\
\indent Now, to prove the statement, we just need to show that the following divisors are not linearly
equivalent to a toric effective divisor.
\begin{itemize}
 \item[a)]$-Z_4$, $-Z_4-Z_7$, $-Z_4-Z_5$, $-Z_4-Z_5-Z_7,$ $-Z_1-Z_4-Z_5,$\linebreak $-Z_1-Z_4-Z_5-Z_7,$
$-Z_4+Z_7,$ $-Z_4-Z_5+Z_7,$
$-Z_1-Z_4-Z_5+Z_7,$\linebreak
$Z_1-Z_4+Z_5-Z_7$, $Z_1-Z_4$, $Z_1-Z_4-Z_7$, $-Z_1-Z_4$, $Z_1+Z_5-Z_4$,\linebreak $Z_1+Z_7-Z_4$, $-Z_4+Z_7$;
\item[b)] $-Z_1-Z_5$, $-Z_1-Z_5-Z_7,$ $-Z_1-Z_5+Z_7,$ $-Z_1+Z_4-Z_5,$ $-Z_5,$ $-Z_5-Z_7$, $-Z_1-Z_5-Z_7$,
$-Z_5+Z_7$;
\item[c)] $-Z_7,$ $-Z_1$, $-Z_1-Z_7$, $Z_7-Z_1$.
\end{itemize}
Looking at \eqref{eq:basis e2} it is obvious that the divisors in a) and b), which have or negative $z_4$ or a
negative  $z_5$ cannot be toric
effective. The divisor in the last group have both $z_4$ and $z_5$ equal to zero, while at least one among
$z_1$ and $z_7$ is equal to -1. Thus it
is straightforward to see that they cannot be linearly equivalent to a toric effective divisor.
\end{proof}
\section{Conclusions}
\indent Collecting all the results we obtained so far we are able to enunce:
\begin{thm}[Main Theorem]\label{mainthm}
All toric Fano 3-folds have a full strongly exceptional collection made up of line bundles.
\end{thm}
\begin{proof}
 As mentioned in the introduction to this paper, there were just five varieties left out: $\duno$, $\ddue$,
$\euno$, $\edue$ and $\equat$. In theorem
\ref{prop:duno} we proved that $\duno$ admits a full strongly exceptional sequence. In theorems
\ref{prop:ddue}, \ref{prop:euno},
\ref{prop:edue} and \ref{prop:equat} we showed that the same thing is true for the other four left.
\end{proof}

\section*{Acknowledgements}

 This work was mainly done when both authors were attending the edition of 2009 of P.R.A.G.MAT.I.C. summer
school in Catania. We would like to thank Laura Costa and Rosa-Maria Mir\'{o}-Roig for having suggested us the
problem we worked with and for the
time 
they dedicated to us giving useful advices and suggestions. Our thanks go, also, to the mathematic and
informatic department of Universit\`a  di
Catania
and the P.R.A.G.MAT.I.C organizers for the opportunity we recieved and the kind hospitality.\\
The second author also thanks Antonio Rapagnetta for very usefull conversation.
\bibliography{Fano_3-Folds.bib}

\end{document}